\documentclass[a4paper,11pt]{amsart}

\usepackage{amsfonts,latexsym,rawfonts,amsmath,amssymb,amsthm, mathrsfs, lscape}
\usepackage[english]{babel}
\usepackage{setspace}
\usepackage{textcomp}
\usepackage{tikz}
\usepackage{xy}
\usepackage{graphicx}
\usepackage[inline]{enumitem}
\usepackage{fullpage}
\usepackage{xfrac}

\usepackage[hypertexnames=false,
backref=page,
    pdftex,
    pdfpagemode=UseNone,
    breaklinks=true,
    extension=pdf,
    colorlinks=true,
    linkcolor=blue,
    citecolor=blue,
    urlcolor=blue,
]{hyperref}

\renewcommand{\Re}{\mathsf{Re}\,}
\renewcommand{\Im}{\mathsf{Im}\,}

\newcommand{\referenza}{}

\newtheorem{thm}{Theorem}[section]
\newtheorem*{thm*}{Theorem \referenza}
\newtheorem{cor}[thm]{Corollary}
\newtheorem*{cor*}{Corollary \referenza}
\newtheorem{lem}[thm]{Lemma}
\newtheorem*{lem*}{Lemma \referenza}

\newtheorem*{prop*}{Proposition \referenza}

\newtheorem*{conj*}{Conjecture \referenza}

\theoremstyle{definition}
\newtheorem{rmk}[thm]{Remark}

\numberwithin{equation}{section}

\def \Q {\mathbb Q}
\def \R {\mathbb R}
\def \C {\mathbb C}
\def \Z {\mathbb Z}

\def\Xint#1{\mathchoice
{\XXint\displaystyle\textstyle{#1}}
{\XXint\textstyle\scriptstyle{#1}}
{\XXint\scriptstyle\scriptscriptstyle{#1}}
{\XXint\scriptscriptstyle\scriptscriptstyle{#1}}
\!\int}
\def\XXint#1#2#3{{\setbox0=\hbox{$#1{#2#3}{\int}$ }
\vcenter{\hbox{$#2#3$ }}\kern-.6\wd0}}

\def\dashint{\Xint-}

\allowdisplaybreaks[1]

\definecolor{DarkViolet}{rgb}{0.58,0.00,0.83}
\newcommand{\dan}[1]{{\color{black}{#1}}}

\newcommand{\rosso}[1]{{\color{black}{#1}}}

\title[Leafwise flat forms on Inoue-Bombieri surfaces]{Leafwise flat forms on Inoue-Bombieri surfaces}

 \author{Daniele Angella}
 \address[Daniele Angella]{
  Dipartimento di Matematica e Informatica ``Ulisse Dini''\\
  Universit\`a degli Studi di Firenze\\
  viale Morgagni 67/a\\
  50134 Firenze, Italy
 }
 \email{daniele.angella@gmail.com}
 \email{daniele.angella@unifi.it}

 \author{Valentino Tosatti}
 \address[Valentino Tosatti]{
 Courant Institute of Mathematical Sciences\\
 New York University\\
 251 Mercer St\\
  New York, NY 10012, USA
 }
 \email{tosatti@cims.nyu.edu}

\makeatletter
\@namedef{subjclassname@2020}{
  \textup{2020} Mathematics Subject Classification}
\makeatother

\thanks{During the preparation of the work, the first-named author has been supported by Project FIRB ``Geometria Differenziale e Teoria Geometrica delle Funzioni'', by Project SIR2014 ``Analytic aspects in complex and hypercomplex geometry'' (code RBSI14DYEB), by Project PRIN ``Variet\`a reali e complesse: geometria, topologia e analisi armonica'' and PRIN2017 ``Real and Complex Manifolds: Topology, Geometry and holomorphic dynamics'' (code 2017JZ2SW5), and by GNSAGA of INdAM. The second-named author was partially supported by NSF grant DMS-1610278, DMS-1903147, DMS-2231783 and by a Chaire Poincar\'e at Institut Henri Poincar\'e}

\date{\today}

\begin{document}

\begin{abstract}
We prove that every Gauduchon metric on an Inoue-Bombieri surface admits a strongly leafwise flat form in its $\partial\overline\partial$-class.
Using this result, we deduce uniform convergence of the normalized Chern-Ricci flow starting at any Gauduchon metric on all Inoue-Bombieri surfaces.
We also show that the convergence is smooth with bounded curvature for initial metrics in the $\partial\overline\partial$-class of the Tricerri/Vaisman metric.
\end{abstract}

\maketitle

\section{Introduction}
In this paper we are interested in the convergence of the normalized Chern-Ricci flow on Inoue-Bombieri surfaces. The {\em Chern-Ricci flow} is a parabolic evolution equation for Hermitian metrics (with associated $(1,1)$-forms) $\omega(t)$ on a compact complex manifold, given by
\begin{equation}\label{eq:CRF0}\tag{CRF}
 \frac{\partial}{\partial t}\omega(t) = -\mathrm{Ric}^{Ch}(\omega(t)), \qquad \omega(0) = \omega_0 ,
\end{equation}
where $\mathrm{Ric}^{Ch}(\omega)\stackrel{\text{loc}}{=}-\sqrt{-1}\,\partial\overline\partial\log\det \omega$ denotes the Chern-Ricci form of $\omega$, and $\omega_0$ is any initial Hermitian metric. It was first studied by M. Gill in the setting of manifolds with vanishing first Bott-Chern class \cite{gill}, where a parabolic proof of the non-K\"ahler Calabi-Yau Theorem \cite{tosatti-weinkove-jams} was given, and then introduced and studied in general by B. Weinkove and the second-named author \cite{tosatti-weinkove-jdg}.
If the initial metric is K\"ahler, then the Chern-Ricci flow equals the K\"ahler-Ricci flow, but in general it is quite different from the Ricci flow.
The behavior of the Chern-Ricci flow on compact complex surfaces was investigated in \cite{tosatti-weinkove-comp, tosatti-weinkove-mathann, fang-tosatti-weinkove-zheng, kawa, nie2, to, E, AS}, and further general results in all dimensions can be found in \cite{gill2, GS,KN,lauret, lauret-rodriguezvalencia, nguyen, nie, SW,yang,zheng}, see also \cite{leetam,HLT} for the Chern-Ricci flow on noncompact complex manifolds, and \cite{tosatti-weinkove-survey} for a survey.
From all these works it is clear that the behavior of solutions of the Chern-Ricci flow deeply reflects the underlying complex structure. Understanding the behavior of the Chern-Ricci flow on non-K\"ahler compact complex surfaces is particularly interesting, due to the fact that such surfaces are not completely classified.

Recall (see {\itshape e.g.} \cite{bhpv}) that minimal non-K\"ahler compact complex surfaces can be divided into three classes according to their Kodaira dimension, namely minimal non-K\"ahler properly elliptic surfaces ($\mathrm{Kod}=1$), Kodaira surfaces ($\mathrm{Kod}=0$), and minimal surfaces of class VII ($\mathrm{Kod}=-\infty$). The behavior of the Chern-Ricci flow on minimal non-K\"ahler properly elliptic surfaces is described in general in \cite{tosatti-weinkove-mathann}, while Kodaira surfaces are covered by \cite{gill}. Minimal class VII surfaces $S$ with $b_2(S)=0$ are also classified \cite{bogomolov-1976, bogomolov, li-yau-zheng0, li-yau-zheng, teleman}, and they are either Hopf or Inoue-Bombieri surfaces, while those with $b_2(S)>0$ are not classified in general, see {\itshape e.g.} \cite{kato, dloussky-oeljeklaus-toma,Te1,Te2}. The Chern-Ricci flow on Hopf surfaces has been studied in \cite{tosatti-weinkove-jdg,tosatti-weinkove-comp,E}, and in this paper we focus on Inoue-Bombieri surfaces.

More precisely, we will consider the {\em normalized Chern-Ricci flow} starting at a Hermitian metric $\omega_0$:
\begin{equation}\label{eq:CRF}\tag{NCRF}
 \frac{\partial}{\partial t}\omega(t) = -\mathrm{Ric}^{Ch}(\omega(t))-\omega(t) , \qquad \omega(0) = \omega_0 ,
\end{equation}
where the underlying manifold will be an Inoue-Bombieri surface.

{\em Inoue-Bombieri surfaces} \cite{Bom,inoue} are surfaces of class VII with second Betti number equal to zero and with no holomorphic curves \cite{bogomolov-1976, bogomolov, li-yau-zheng0, li-yau-zheng, teleman}. Their universal cover is $\C\times\mathbb H$, where $\mathbb H$ denotes the upper half-plane.
They are divided into three families: $S_M$, $S^+_{N,p,q,r;\mathbf{t}}$, and $S^-_{N,p,q,r}$.
They have a structure of fibre bundle over $\mathbb S^1$, where the fibre is a $3$-dimensional torus in case $S_M$, and a compact quotient of the $3$-dimensional Heisenberg group in case $S^\pm$. Furthermore, every Inoue-Bombieri surface of type $S^-$ has an unramified double cover of type $S^+$.

On any Inoue-Bombieri surface $S$, the standard metric on $\mathbb H$ (with coordinate $w=x_2+\sqrt{-1}\,y_2$) induces the degenerate metric \cite{harvey-lawson}
\begin{equation}\label{eq:alpha}
\alpha := \frac{\sqrt{-1}}{4y_2^2}dw\wedge d\bar w ,
\end{equation}
which satisfies
$$ 0 \leq \lambda \alpha \in -c_1^{BC}(S) , $$
where $\lambda=1$ when $S$ is of type $S_M$ and $\lambda=2$ when $S$ is of type $S^{\pm}$. For convenience, let us define
\begin{equation}\label{eq:alpha2}
\omega_\infty := \lambda\alpha .
\end{equation}
Moreover, the kernel of $\alpha$ defines a holomorphic foliation on $S$ by parabolic Riemann surfaces, whose leaves are dense in the fibres of the bundle structure.

It follows from \cite[Theorem 1.2]{tosatti-weinkove-jdg} (cf. \cite[Theorem 2.1]{tosatti-weinkove-mathann}) that the Chern-Ricci flow starting at any Hermitian metric on an Inoue-Bombieri surface has a unique solution for all positive time.
In \cite[\S5, \S6, \S7]{tosatti-weinkove-comp}, explicit solutions of the normalized Chern-Ricci flow starting at the Tricerri, respectively Vaisman-Tricerri metric on an Inoue-Bombieri surface $S_M$, respectively $S^\pm$, were shown to converge to $\mathbb S^1$ in the sense of Gromov-Hausdorff, \cite[Theorem 5.1, Theorem 6.1, Theorem 7.1]{tosatti-weinkove-comp}. More generally, this holds for all initial \dan{locally} homogeneous metrics, \cite{lauret, lauret-rodriguezvalencia}.

This convergence result is extended to a larger class of initial metrics in \cite{fang-tosatti-weinkove-zheng}, where S. Fang, the second-named author, B. Weinkove, and T. Zheng proved that the normalized Chern-Ricci flow collapses any Hermitian metric on an Inoue-Bombieri surface to a circle, modulo an initial conformal change.
In fact, they proved that on an Inoue-Bombieri surface the solution of the normalized Chern-Ricci flow starting at any Hermitian metric in the $\partial\overline\partial$-class of a $(1,1)$-form which is {\em strongly flat along the leaves} converges uniformly to $\omega_\infty$ as $t\to+\infty$, \cite[Theorem 1.1]{fang-tosatti-weinkove-zheng}, and that this implies that the Gromov-Hausdorff limit is a circle.
The convergence is in fact in $\mathcal{C}^\beta$, for every $0<\beta<1$, when the initial metric is in the $\partial\overline\partial$-class of the Tricerri, respectively Vaisman-Tricerri metric, \cite[Theorem 1.3]{fang-tosatti-weinkove-zheng}.
Moreover, any Hermitian metric on an Inoue-Bombieri surface admits a Hermitian metric in its conformal class which is strongly flat along the leaves \cite{fang-tosatti-weinkove-zheng}.

Here, a real $(1,1)$-form $\omega$ (not necessarily a Hermitian metric) on an Inoue-Bombieri surface $S$ is called {\em flat along the leaves} if the restriction of $\omega$ to every leaf of the holomorphic foliation of $S$ is a flat K\"ahler metric on $\C$ in case $S_M$, respectively $\C^*$ in case $S^\pm$. Equivalently, consider the universal cover $P\colon \C\times\mathbb H \to S$. Then $\omega$ is flat along the leaves if and only if $P^*\omega\lfloor_{\C\times\{w\}}$ is a flat K\"ahler metric for any $w\in\mathbb H$. This is equivalent to asking that $\alpha\wedge\omega=\pi^*\eta\,\omega_{TV}^2$, for some $\eta\in\mathcal{C}^\infty(\mathbb S^1;\R^{>0})$, where $\pi\colon S_M\to \mathbb S^1$ denotes the projection of the bundle structure, \cite[Lemma 2.1]{fang-tosatti-weinkove-zheng} and $\omega_{TV}$ denotes the Tricerri metric \cite{tricerri} in the case $S_M$, respectively the Vaisman-Tricerri metric \cite{vaisman,tricerri} in the case $S^\pm$.
If moreover $P^*\omega\lfloor_{\C\times\{w\}}$ is equal to $c(\Im w)\cdot\sqrt{-1}\,dz\wedge d\bar z$ on $S_M$, respectively $c\cdot \sqrt{-1}\, dz\wedge d\bar z$ on $S^\pm$, (here $z$ is the coordinate on $\C$), then $\omega$ is called {\em strongly flat along the leaves}.
This is equivalent to asking that $\alpha\wedge\omega=c\,\omega_{TV}^2$, where $c>0$ is a constant, \cite[Lemma 2.1]{fang-tosatti-weinkove-zheng}.

The results in \cite{fang-tosatti-weinkove-zheng} left open the question of the behavior of the Chern-Ricci flow on Inoue-Bombieri surfaces when the initial Hermitian metric is arbitrary. More precisely in Question 1 in \cite[\S4]{fang-tosatti-weinkove-zheng} they asked whether all Hermitian metrics belong to the $\partial\overline\partial$-class of a $(1,1)$-form which is strongly flat along the leaves, or whether this holds at least for Gauduchon metrics. In this paper we answer these questions.

Our first observation, see Lemma \ref{lem:ostruction-slf}, is that it is in fact not true that all Hermitian metrics on an Inoue-Bombieri surface belong to the $\partial\overline\partial$-class of a $(1,1)$-form which is strongly flat along the leaves. This follows from a simple obstruction (see \eqref{obstrukt} below) coming from elements in the kernel of the leafwise Laplacian. We also observe that this obstruction vanishes for all Gauduchon metrics.

More interestingly, our main Theorems \ref{main2} and \ref{main3} show that if this obstruction vanishes (in particular this holds for all Gauduchon metrics) then the Hermitian metric does belong to the $\partial\overline\partial$-class of a $(1,1)$-form which is strongly flat along the leaves. This gives:

\begin{thm}\label{thm:main-thm}
 Let $S$ be an Inoue-Bombieri surface. Let $\omega$ be a Gauduchon metric on $S$ (or more generally a Hermitian metric which satisfies \eqref{obstrukt}). Then there exists a smooth function $u$ on $S$ such that $\omega+\sqrt{-1}\,\partial\overline\partial u$ is a real $\partial\overline{\partial}$-closed $(1,1)$-form which is strongly flat along the leaves.
\end{thm}

This theorem reduces to solving the degenerate elliptic equation
\begin{equation}\tag{\ref{eq:Deltau=G}}
 \Delta_{\mathcal D}u = G(\omega) ,
\end{equation}
where
$$ \Delta_{\mathcal{D}}u := \frac{\sqrt{-1}\,\partial\overline\partial u\wedge\alpha}{\omega_{TV}^2} $$
is the Laplacian along the leaves,
and where we set
$$ G(\omega) := -\frac{\omega\wedge\alpha}{\omega_{TV}^2}+\frac{\int_{S}\omega\wedge\alpha}{\int_{S}\omega_{TV}^2} . $$

A necessary condition for the solvability of \eqref{eq:Deltau=G} is given by
\begin{equation}\label{obstrukt}
 G(\omega) \perp_{L^2(\omega_{TV}^2)} \ker\Delta_{\mathcal{D}} ,
\end{equation}
and this is satisfied by Gauduchon metrics $\omega$, Lemma \ref{lem:ostruction-slf}. Inoue-Bombieri surfaces are bundles over the circle with fiber $\mathbb{T}^3$ in the case of $S_M$ and a $3$-dimensional nilmanifold in the case of $S^{\pm}$. Using Fourier expansion along these fibers (partial Fourier expansion in the case of nilmanifolds), we obtain a distributional solution to \eqref{eq:Deltau=G} whenever \eqref{obstrukt} is satisfied, and we show that this solution is in fact smooth using crucially the Liouville theorem on rational approximations of irrational algebraic numbers.

As a consequence of Theorem \ref{thm:main-thm}, using more or less directly \cite[Theorem 1.1]{fang-tosatti-weinkove-zheng}, we get uniform convergence for the normalized Chern-Ricci flow starting at any Gauduchon metric on all Inoue-Bombieri surfaces, thus answering a question in \cite[Conjecture 1, page 3183]{fang-tosatti-weinkove-zheng}.
\begin{cor}\label{kor}
 Let $S$ be an Inoue-Bombieri surface, and $\omega$ be any Gauduchon metric on $S$ (or more generally a Hermitian metric satisfying \eqref{obstrukt}). Let $\omega(t)$ be the solution of the normalized Chern-Ricci flow \eqref{eq:CRF} starting at $\omega$. Then
 $$ \omega(t) \to \omega_\infty \qquad \text{ as } t\to+\infty , $$
uniformly on $S$ and exponentially fast, where $\omega_\infty$ is defined in \eqref{eq:alpha2}.
 Moreover,
 $$ \left( S, \omega(t) \right) \to \left( \mathbb{S}^1, d \right) \qquad \text{ as } t\to+\infty $$
 in the Gromov-Hausdorff sense, where $d$ is the standard metric on $\mathbb{S}^1$, of radius depending on $S$.
\end{cor}

\bigskip

In \cite[Conjecture 2, page 3183]{fang-tosatti-weinkove-zheng} it was also conjectured that in the setting of Corollary \ref{kor}, the metrics $\omega(t)$ converge to $\omega_\infty$ smoothly, and it was suggested that this could be first approached in the case of Gauduchon metrics in the $\partial\overline\partial$-class of $\omega_{TV}$. Our final theorem confirms the conjecture in this case, and also shows that the evolving metrics $\omega(t)$ collapse to $\omega_\infty$ with uniformly bounded curvature. Our arguments also apply to the setting of non-K\"ahler minimal properly elliptic surfaces as studied in \cite{tosatti-weinkove-comp,tosatti-weinkove-mathann} with initial metrics in the $\partial\overline\partial$-class of the Vaisman metric \cite{vaisman}, which greatly improves the main result of \cite{kawa}.

\begin{thm}\label{thm:main-thm2}
 Let $S$ be an Inoue-Bombieri surface or a non-K\"ahler minimal properly elliptic surface, and let $\omega_{TV}$ be the Tricerri/Vaisman metric on $S$ from \cite{tricerri,vaisman}. Let $\omega$ be a Gauduchon metric on $S$ which is of the form $\omega=\omega_{TV}+\sqrt{-1}\,\partial\overline\partial \psi$ for some smooth function $\psi$.  Let $\omega(t)$ be the solution of the normalized Chern-Ricci flow \eqref{eq:CRF} starting at $\omega$. Then we have that $\omega(t)\to\omega_\infty$ in the $C^\infty$ topology, and furthermore
 $$\sup_S |\mathrm{Rm}(\omega(t))|_{\omega(t)}\leq C,$$
 for all $t\geq 0$.
\end{thm}

The main idea that we will use originates from the work of Gross-Tosatti-Zhang \cite{GTZ} on collapsing Calabi-Yau manifolds fibered by abelian varieties, with the role of the abelian varieties fibers now played by the leaves of the canonical foliation on $S$, and adapted to the K\"ahler-Ricci flow in \cite[\S 5.14]{To} (see also \cite{FZ,HT,TZ}). We apply a family of stretchings in the direction of the leaves to make the PDE uniformly elliptic, and using the explicit behavior of the Tricerri/Vaisman metric under this stretching we obtain higher order estimates for $\omega(t)$ after stretching from standard higher-order regularity of uniformly elliptic PDEs of complex Monge-Amp\`ere type.

\bigskip

The paper is organized as follows.
In Sections \ref{sec:inoue-sm} and \ref{sec:inoue-s+-}, we recall the construction of Inoue-Bombieri surfaces of type $S_M$, respectively $S^\pm$, and their main properties. In Section \ref{sec:gauduchon}, we study Gauduchon metrics on Inoue-Bombieri surfaces, and show that a certain natural obstruction for constructing strongly leafwise flat forms vanishes for such metrics (Lemmas \ref{lem:gaud-sm} and \ref{lem:gauduchon-s+-}). In Section \ref{sec:slf-metrics} we prove our main result, Theorem \ref{thm:main-thm} \dan{(see Theorems \ref{main2} and \ref{main3} respectively)}, by solving a degenerate elliptic equation on these surfaces, whenever the aforementioned obstruction vanishes, and we also deduce Corollary \ref{kor}. And lastly in Section \ref{sec:higher} we give the proof of Theorem \ref{thm:main-thm2}.

\bigskip

{\small
\noindent{\sl Acknowledgments.}
We are grateful to Gian Maria Dall'Ara, Serena Matucci, Fulvio Ricci, Ben Weinkove, and Steve Zelditch for very helpful discussions. The first-named author would like to thank for the warm hospitality during his stay at the Department of Mathematics of the Northwestern University, where this work was originally conceived.
This work was partially written during the first-named author's visit to the Facultad de Matem\'aticas of the Universidad Complutense de Madrid, and during the second-named author's visits to the Center for Mathematical Sciences and Applications at Harvard University and to the Institut Henri Poincar\'e in Paris (supported by a Chaire Poincar\'e), which we would like to thank for the hospitality.
}

\section{The Inoue-Bombieri surfaces \texorpdfstring{$S_M$}{SM}}\label{sec:inoue-sm}

\subsection{Construction}
Consider the Inoue-Bombieri surface \cite{inoue}
$$ S_M := \left. (\C\times\mathbb H) \middle\slash \Gamma \right. , $$
with coordinates $z=x_1+\sqrt{-1}y_1\in\C$ and $w=x_2+\sqrt{-1}y_2\in\mathbb H$, with $y_2>0$.

Here, we are given a matrix $M=(M_{jk})_{j,k}\in \mathrm{SL}(3;\Z)$ with a real eigenvalue $\lambda>1$ and complex non-real eigenvalues $\mu$ and $\bar\mu$. Note that $\lambda\in\R\setminus\Q$, (indeed, $\lambda\neq 1$ is a root of a monic polynomial with $0$th order term equal to $1$), and $\lambda|\mu|^2=1$. Denote by $(\ell_1,\ell_2,\ell_3)$ an eigenvector for $\lambda$, and by $(m_1,m_2,m_3)$ an eigenvector for $\mu$.
Since $\lambda$ is irrational, it follows immediately that at least two of the $\ell_j$'s must be nonzero, so at least one of the ratios $\{\ell_i/\ell_j\}_{i\neq j}$ is well-defined and is an irrational algebraic number.

Define $\Gamma=\left\langle f_0,f_1,f_2,f_3 \right\rangle$ to be the subgroup of automorphisms of $\C\times\mathbb H$ generated by
$$ f_0(z,w):=(\mu z,\lambda w),\qquad f_j(z,w):=(z+m_j,w+\ell_j), $$
varying $j\in\{1,2,3\}$.
The action of $\Gamma$ on $\C\times\mathbb H$ is fixed-point free and properly discontinuos with compact quotient, so $S_M$ is a compact complex manifold.
Denote by $P\colon \C\times\mathbb H \to S_M$ the projection.

\subsection{Torus-bundle structure}
Note that $S_M$ has a structure of $\mathbb T^3$-bundle over $\mathbb S^1$, with projection
$$ \pi \colon S_M \ni (z,w) \mapsto \Im w \in \left. \R^{>0} \middle\slash \left\langle y_2\mapsto \lambda y_2\right\rangle \simeq \mathbb{S}^1 \right.. $$

More precisely, notice that $\left\{ (m_1,m_2,m_3), \, (\overline m_1,\overline m_2,\overline m_3), \, (\ell_1,\ell_2,\ell_3) \right\}$ are $\C$-linearly independent, whence $$\left\{(\Re m_1, \Re m_2, \Re m_3), \, (\Im m_1,\Im m_2,\Im m_3), \, (\ell_1,\ell_2,\ell_3) \right\}$$ are $\R$-linearly independent.
Therefore the subgroup $\Gamma^\prime := \left\langle f_1,f_2,f_3 \right\rangle < \Gamma$ is isomorphic to $\Z^3$. Moreover it acts properly-discontinuosly and freely on $\C\times\mathbb H$, with quotient $\tilde X := \left. (\C\times\mathbb H) \middle\slash \Gamma^\prime \right. \simeq \mathbb T^3\times\R^{>0}$. The projection $\pi_1\colon \tilde X \to \R^{>0}$ is induced by $(z,w)\mapsto \Im w$.

Since $\mu\cdot m_j=\sum_{k=1}^{3} M_{jk}\cdot m_k$ and $\lambda\cdot \ell_j=\sum_{k=1}^{3} M_{jk}\cdot \ell_k$, with $M_{jk}\in\Z$, then $f_0$ descends to a map $f_0\colon \tilde X \to \tilde X$, and we have that $S_M = \tilde X \slash \langle f_0 \rangle$.

Note that $f_0$ maps the torus fibre $T_y:=\pi^{-1}(y)$ to the torus fibre $T_{\lambda\cdot y}:=\pi^{-1}(\lambda\cdot y)$. In particular, $f_0$ induces the diffeomorphism $\Psi\colon T_1 \stackrel{\simeq}{\to} T_{\lambda}$. We get that $S_M$ has a structure of mapping torus as follows:
$$ S_M \simeq \left. \left(  \mathbb T^3 \times [1,\lambda] \right) \middle\slash \left( (p,1) \sim ( \Psi(p),\lambda ) \right) \right. . $$

\subsection{Foliation}
The form $\alpha := \frac{\sqrt{-1}}{4y_2^2}dw\wedge d\bar w$ being closed, induces a foliation on $S_M$, whose leaves are dense in the torus fibre.

More precisely, the kernel of $\frac{\sqrt{-1}}{4y_2^2}dw\wedge d\bar w$ on $\C \times \mathbb H$ is the integrable distribution
$$ \mathcal{D} := \mathrm{span}_{\C} \left\{ \frac{\partial}{\partial z} \right\} ,$$
whose leaves are of the form
$$ \mathcal{L}_{w_0} := \left\{ (z,w_0) \in \C\times\mathbb H : z \in \C \right\} , w_0\in\mathbb H.$$
It induces a holomorphic foliation $\mathcal{D}$ on $S_M$, without singularities, and whose leaves $P(\mathcal{L}_{w_0})$ are biholomorphic to $\C$, and Inoue \cite{inoue} (see also \cite[Proposition V.19.1]{bhpv} for a simple proof) showed that under the projection $P \colon \C\times\mathbb H \to S_M$, for any $w_0\in\mathbb H$, the image $P(\mathcal{L}_{w_0})$ is dense in the torus fibre $T:=\pi^{-1}(\Im w_0)\subset S_M$.

\subsection{Solvmanifold structure}
The Inoue-Bombieri surface has a structure of solvmanifold with invariant complex structure, discovered in \cite{Wa}.

\dan{
We consider the holomorphic action of the solvable Lie group $(\mathbb C \times \mathbb R) \rtimes \mathbb R$ on $\mathbb C \times \mathbb H$ given by
\begin{eqnarray*}
\lefteqn{ \left(x_1+\sqrt{-1}y_1,x_2+\sqrt{-1}y_2\right) \stackrel{\left(\left(a+\sqrt{-1}b,s\right),t\right)}{\mapsto}}\\
&&
\left( (\mu^t x_1+a)+\sqrt{-1}(\mu^t y_1+b), (\lambda^t x_2+s)+\sqrt{-1}\lambda^t y_2 \right).
\end{eqnarray*}
We get an action of the lattice $\Gamma:=\mathbb Z^3 \rtimes_M \mathbb Z$ on $\mathbb C \times \mathbb H$, and then
$$ S_M = \mathbb C \times \mathbb H \slash \Gamma . $$
}

\dan{The forms on $\mathbb C\times\mathbb H$ defined by
$$
 e_1 := \frac{1}{\sqrt{y_2}}\, \frac{\partial}{\partial x_1} , \qquad
 e_2 := \frac{1}{\sqrt{y_2}}\, \frac{\partial}{\partial y_1} ,
\qquad
 e_3 := y_2\, \frac{\partial}{\partial x_2} , \qquad
 e_4 := y_2\, \frac{\partial}{\partial y_2} ,
$$
are invariant, and give a global frame on $S_M$. Its dual co-frame is
}
$$
 e^1 := \sqrt{y_2}\, d x_1 , \qquad
 e^2 := \sqrt{y_2}\, d y_1 ,
 \qquad
 e^3 := \frac{1}{y_2}\, d x_2 , \qquad
 e^4 := \frac{1}{y_2}\, d y_2 .
$$
The structure equations are
$$
 de^1 = - \frac{1}{2}e^{1}\wedge e^{4}, \qquad
 de^2 = - \frac{1}{2}e^{2}\wedge e^{4},
 \qquad
 de^3 = e^{3}\wedge e^{4}, \qquad
 de^4 = 0 .
$$

The complex structure on $S_M$ induced by the invariant complex structure on $\C\times\mathbb H$ is given by
$$ Je_1 := e_2 , \qquad Je_3 := e_4 . $$

So a global co-frame of $(1,0)$-forms on $S_M$ is given by
$$ \varphi^1 := e^1+\sqrt{-1}e^2 = \sqrt{y_2}\, dz , \qquad \varphi^2 := e^3+\sqrt{-1}e^4 = \frac{1}{y_2}\, dw , $$
with structure equations
$$ d\varphi^1 = \frac{\sqrt{-1}}{4}\varphi^1\wedge\varphi^2-\frac{\sqrt{-1}}{4}\varphi^1\wedge\bar\varphi^2 ,
\qquad d\varphi^2=\frac{\sqrt{-1}}{2}\varphi^2\wedge\bar\varphi^2 . $$

\subsection{Tricerri metric}
Consider the degenerate metric \cite{harvey-lawson}
\begin{equation}\tag{\ref{eq:alpha}}
\alpha := \frac{\sqrt{-1}}{4y_2^2}dw\wedge d\bar w = \frac{\sqrt{-1}}{4}\varphi^2\wedge\bar\varphi^2 ,
\end{equation}
namely, the pull-back of the Poincar\'e metric on $\mathbb H$.

Consider the Tricerri metric \cite{tricerri}
$$ \omega_T := 4\alpha+\sqrt{-1}\varphi^1\wedge\bar\varphi^1 = \sqrt{-1}\varphi^1\wedge\bar\varphi^1 + \sqrt{-1}\varphi^2\wedge\bar\varphi^2 .$$
Note that the Tricerri metric is an lcK metric on $S_M$, with Lee form
$$ \vartheta := \frac{1}{2}\varphi^2-\frac{1}{2}\bar\varphi^2 .$$
Notice that $S_M$ does not contain any curve, as proven in \cite{inoue}.
Indeed, since $b_2(S_M)=0$, the closed $(1,1)$-form $\alpha$ is exact. If there were a curve $C$, then $\int_C\alpha=0$. Since $\alpha\geq0$, then $TC\subseteq\mathcal{D}$, that is, $C$ is contained in a leaf of $\mathcal{D}$, which is dense in a torus $\mathbb T^3$ as mentioned above. This is impossible.

\section{The Inoue-Bombieri surfaces \texorpdfstring{$S^\pm$}{S+-}}\label{sec:inoue-s+-}

\subsection{Construction}
Consider the Inoue-Bombieri surface
$$ S^+ := S^+_{N,p,q,r,\mathbf{t}} := \left. (\C\times\mathbb H) \middle\slash \Gamma \right. , $$
with coordinates $z=x_1+\sqrt{-1}y_1\in\C$ and $w=x_2+\sqrt{-1}y_2\in\mathbb H$, with $y_2>0$.
Here, we choose $N=(n_j^k)_{j,k}\in \mathrm{SL}(2;\Z)$ with two real eigenvalues $\gamma$ and $\gamma^{-1}$ with $\gamma>1$ (it follows that $\gamma$ is irrational), and $p,q,r\in\Z$ such that $r\neq0$, and $\mathbf t\in\C$.
Let $(a_1,a_2)$ and $(b_1,b_2)$ be two real eigenvectors corresponding to the eigenvalues $\gamma$ and $\gamma^{-1}$.
Let $(c_1,c_2)\in\R^2$ be a solution of the linear equation
\begin{equation}\label{eq:eq-for-c}
(c_1,c_2) = (c_1,c_2) \cdot N^t + (e_1,e_2) + \frac{b_1\cdot a_2-b_2\cdot a_1}{r}\cdot (p,q) ,
\end{equation}
where, for $j\in\{1,2\}$,
$$ e_j :=
 \frac12\cdot n_{j1}\cdot (n_{j1}-1)\cdot a_1 \cdot b_1
 + \frac12\cdot n_{j2}\cdot (n_{j2}-1)\cdot a_2 \cdot b_2
 + n_{j1}\cdot n_{j2}\cdot b_{1}\cdot a_{2} .
$$
Whence define $\Gamma=\left\langle f_0,f_1,f_2,f_3 \right\rangle$ to be the subgroup of automorphisms of $\C\times\mathbb H$ generated by
$$ f_0(z,w):=(z+\mathbf t,\gamma\cdot w),\qquad f_j(z,w):=(z+b_j\cdot w+c_j,w+a_j), $$
varying $j\in\{1,2\}$, and
$$ f_3(z,w) := \left( z+\frac{b_1\cdot a_2-b_2\cdot a_1}{r}, w\right) .$$
The action of $\Gamma$ on $\C\times\mathbb H$ is fixed-point free and properly discontinuous with compact quotient, so $S^+$ is a compact complex manifold.

The Inoue-Bombieri surface $S^-$ is constructed similarly (see \cite{inoue}), but the important point for us is that it has an unramified double cover of type $S^+$.

\begin{rmk}\label{irraz}
For later use, let us remark here that neither the slope of $(a_1,a_2)$ nor the slope of $(b_1,b_2)$ (which are clearly algebraic numbers) can be rational numbers.
If this was true, then $(1, a_2/a_1)$ (or $(a_1/a_2,1)$), \dan{or} $(1,b_2/b_1)$ (or $(b_1/b_2,1)$) would be eigenvectors of $N=(n_{jk})_{j,k}$ with rational entries. This means that, for example, $n_{11}+n_{12}\frac{a_2}{a_1}=\gamma$ would be rational.
\end{rmk}

\subsection{Nilmanifold-bundle structure}
On $S^+$, as in the case of $S_M$, we consider the subgroup $\Gamma^\prime := \left\langle f_1,f_2,f_3 \right\rangle < \Gamma$, which acts properly-discontinuosly and freely on $\C\times\mathbb H$, with quotient $\tilde X :=  (\C\times\mathbb H) \slash \Gamma^\prime$ which is described as follows:
\begin{equation}\label{quot}
\tilde{X}= \left\{ \left(\begin{array}{ccc}1&y_1&x_1\\&y_2&x_2\\&&1\end{array}\right):x_1,y_1,x_2\in \R, y_2\in\R^{>0}\right\}/\Gamma'',
\end{equation}
where $\Gamma''\simeq\Gamma'$ is the discrete group generated by the matrices
$$g_1=\left(\begin{array}{ccc}1&b_1&c_1\\&1&a_1\\&&1\end{array}\right),\quad g_2=\left(\begin{array}{ccc}1&b_2&c_2\\&1&a_2\\&&1\end{array}\right),\quad g_3=\left(\begin{array}{ccc}1&0&c_3\\&1&0\\&&1\end{array}\right),$$
acting by left multiplication, where
$$ c_3 = \frac{b_1\cdot a_2-b_2\cdot a_1}{r} . $$
Indeed, setting $z=x_1+\sqrt{-1}y_1$ and $w=x_2+\sqrt{-1}y_2$, the action of the matrix $g_j$ is identical to the action of $f_j$, for $j=1,2,3.$  The projection $\pi_1\colon \tilde X \to \R^{>0}$ is induced by $(z,w)\mapsto \Im w$, with fibers $X_{y_2}=\pi_1^{-1}(y_2)$, and we have that $\tilde{X}$ is diffeomorphic to $X\times\R^{>0}$ where $X=X_1$ is the compact $3$-dimensional nilmanifold
$$X_1= H(3;\R) \slash\Gamma''= \left\{ \left(\begin{array}{ccc}1&y_1&x_1\\&1&x_2\\&&1\end{array}\right)=:(x_2,y_1,x_1):x_1,y_1,x_2\in \R\right\}/\Gamma'',$$
where $H(3;\R)$ is the Heisenberg group.
Since $f_0$ lies in the normalizer of $\Gamma'$ by \cite[p. 276]{inoue}, it follows that $f_0$ induces a diffeomorphism $\Psi:X_1\stackrel{\simeq}{\to} X_{\gamma}$ and we have
$$S^+ = \tilde X \slash \langle f_0 \rangle
\simeq X \times [1,\gamma]\slash \left( (p,1)\sim(\Psi(p),\gamma) \right) .$$
The projection
$$ \pi \colon S^+ \ni (z,w) \mapsto y_2 \in \left. \R^{>0} \middle\slash \left\langle y_2\mapsto \gamma\cdot y_2\right\rangle \simeq \mathbb{S}^1 \right. $$
yields a structure of $X$-bundle over $\mathbb S^1$.

\subsection{Foliation}
The form $\frac{\sqrt{-1}}{4y_2^2}dw\wedge d\bar w$ being closed, it induces a foliation on $S^+$, whose leaves are dense in the $X$-fibre.

More precisely, the kernel of $\frac{\sqrt{-1}}{4y_2^2}dw\wedge d\bar w$ on $\C \times \mathbb H$ is the integrable distribution
$$ \mathcal{D} := \mathrm{span}_{\C} \left\{ \frac{\partial}{\partial z} \right\} ,$$
whose leaves are of the form
$$ \mathcal{L}_{w_0} := \left\{ (z,w_0) \in \C\times\mathbb H : z \in \C \right\} , w_0\in\mathbb H.$$
It induces a holomorphic foliation $\mathcal{D}$ on $S^+$, without singularities, and whose leaves $P(\mathcal{L}_{w_0})$ are biholomorphic to $\C^*$, and under the projection $P \colon \C\times\mathbb H \to S^+$, for any $w_0\in\mathbb H$, the image $P(\mathcal{L}_{w_0})$ is dense in the $X$-fibre $\pi^{-1}(\Im w_0)\subset S^+$ (see {\itshape e.g.} \cite[Lemma 6.2]{tosatti-weinkove-comp} or \cite[Proposition 2.1]{Br}).
Indeed, if they were not dense, then $(1,0,0)$ and $(0,1,0)$ would be elements of the lattice generated by $(a_1,b_1,c_1)$, $(a_2,b_2,c_2)$, and $(0,0,c_3)$, which would contradict Remark \ref{irraz}.

\subsection{Solvmanifold structure} See \cite{hasegawa}.
Write $\Im \mathbf{t}=m\log\gamma,$ for some $m\in\mathbb{R}$, and
consider the global co-frame on $S^+$ given by
$$
 e^1 := dx_1-\frac{y_1-m\log y_2}{y_2}\, dx_2 , \qquad
 e^2 := dy_1-\frac{y_1-m\log y_2}{y_2}\, dy_2 ,
$$
$$
 e^3 := \frac{1}{y_2}\, d x_2 , \qquad
 e^4 := \frac{1}{y_2}\, d y_2 , \qquad
$$

In terms of the corresponding dual frame:
$$ e_1 = \frac{\partial}{\partial x_1}, \qquad
e_3 = (y_1-m\log y_2)\frac{\partial}{\partial x_1}+y_2\frac{\partial}{\partial x_2} , $$
$$ e_2 = \frac{\partial}{\partial y_1}, \qquad
e_4 = (y_1-m\log y_2)\frac{\partial}{\partial y_1}+y_2\frac{\partial}{\partial y_2} . $$

The structure equations are
$$
 de^1 = - e^{2}\wedge e^{3} - m\, e^{3}\wedge e^{4}, \qquad
 de^2 = - e^{2}\wedge e^{4},
$$
$$
 de^3 = e^{3}\wedge e^{4}, \qquad
 de^4 = 0 .
$$

The complex structure is
$$ Je_1 := e_2 , \qquad Je_3 := e_4 , $$
in terms of the frame; in terms of the co-frame, $Je^1=-e^2$, $Je^3=-e^4$, so a global co-frame of $(1,0)$-forms on $S^+$ is given by
$$ \varphi^1 := e^1+\sqrt{-1}e^2 = dz - \frac{y_1-m\log y_2}{y_2}\, dw ,$$
$$ \varphi^2 := e^3+\sqrt{-1}e^4 = \frac{1}{y_2}\, dw , $$
with structure equations
$$
d\varphi^1 = \frac{\sqrt{-1}}{2}\varphi^1\wedge\varphi^2+\frac{\sqrt{-1}}{2}\varphi^2\wedge\bar\varphi^1-m\,\frac{\sqrt{-1}}{2}\varphi^2\wedge\bar\varphi^2,
$$
$$
d\varphi^2=\frac{\sqrt{-1}}{2}\varphi^2\wedge\bar\varphi^2 .
$$

\subsection{Vaisman-Tricerri metric}
Consider the degenerate metric
\begin{equation}\tag{\ref{eq:alpha}}
 \alpha := \frac{\sqrt{-1}}{4y_2^2}dw\wedge d\bar w = \frac{\sqrt{-1}}{4}\varphi^2\wedge\bar\varphi^2 .
\end{equation}

Consider the Vaisman-Tricerri metric
$$ \omega_V := 4\alpha+\sqrt{-1}\varphi^1\wedge\bar\varphi^1 = \sqrt{-1}\varphi^1\wedge\bar\varphi^1 + \sqrt{-1}\varphi^2\wedge\bar\varphi^2 .$$
This was discovered in \cite{tricerri} when $\mathbf{t}$ is real and in \cite{vaisman} in general. This metric is lcK if and only if $\mathbf{t}$ is real.

\section{Gauduchon metrics on Inoue-Bombieri surfaces}\label{sec:gauduchon}

Consider an arbitrary Hermitian metric on an Inoue-Bombieri surface $S$ \dan{(either of type $S_M$ or $S^+$)}, which must be of the form
\begin{equation}\label{eq:generic-metric}
   \omega := \sqrt{-1}r\, \varphi^1\wedge\bar\varphi^1 + \sqrt{-1}s\, \varphi^2\wedge\bar\varphi^2
   + u\, \varphi^1\wedge\bar\varphi^2 - \bar u\, \varphi^2\wedge\bar\varphi^1
\end{equation}
where $r,s\in\mathcal{C}^\infty(S;\R)$ and $u\in\mathcal{C}^\infty(S;\C)$ are such that $r>0$, $s>0$, and $rs-|u|^2>0$, at every point.
\dan{In this section, we investigate the {\em Gauduchon condition} for $\omega$, namely,
$$ \partial\overline\partial\omega=0 . $$
}

\subsection{Gauduchon metrics on Inoue-Bombieri surfaces of type \texorpdfstring{$S_M$}{SM}}
If $\omega$ is Gauduchon, then
$$ d\int_{\pi^{-1}(y_2)} \overline\partial \omega = 0 , $$
therefore the function $\tau(y_2):=\int_{\pi^{-1}(y_2)} \overline\partial \omega$ is constant on $\mathbb S^1$.

We compute
\begin{eqnarray*}
   \overline\partial\omega &=&
   \left( \sqrt{-1}\cdot y_2 \cdot \partial_{\bar w} r - \frac{1}{2}\cdot r - \frac{1}{\sqrt{y_2}}\cdot \partial_{\bar z}u \right) \, \varphi^1\wedge\bar\varphi^1\wedge\bar\varphi^2 \\
   && + \left( -\frac{\sqrt{-1}}{\sqrt{y_2}}\partial_{\bar z}s - y_2\cdot \partial_{\bar w}\bar u + \frac{\sqrt{-1}}{4}\cdot \bar u \right) \, \varphi^2\wedge\bar\varphi^1\wedge\bar\varphi^2 .
\end{eqnarray*}
Therefore, \dan{by straightforward computations,
\begin{eqnarray*}
 \tau(y_2)
 &=& \int_{\pi^{-1}(y_2)} \left( 2\cdot y_2 \cdot \partial_{\bar w} r + \sqrt{-1}\cdot r + \frac{2\sqrt{-1}}{\sqrt{y_2}}\cdot \partial_{\bar z}u \right)\, e^{1}\wedge e^{2} \wedge e^{3} \\
 &=& \int_{\pi^{-1}(y_2)} \left( y_2 \cdot \partial_{x_2} r - \frac{1}{\sqrt{y_2}}\cdot \partial_{y_1}u \right)\, e^{1}\wedge e^{2} \wedge e^{3} \\
 && + \sqrt{-1}\cdot \int_{\pi^{-1}(y_2)} \left( y_2 \cdot \partial_{y_2} r + r + \frac{1}{\sqrt{y_2}}\cdot \partial_{x_1}u \right)\, e^{1}\wedge e^{2} \wedge e^{3} \\
&=&
\int_{\pi^{-1}(y_2)} dr\wedge e^1\wedge e^2 + \int_{\pi^{-1}(y_2)} du \wedge e^{1}\wedge e^{3} \\
&& +\sqrt{-1}\cdot \int_{\pi^{-1}(y_2)} \left( y_2 \cdot \partial_{y_2} r + r \right)\, e^{1}\wedge e^{2} \wedge e^{3}+ \int_{\pi^{-1}(y_2)} du\wedge e^2 \wedge e^3 \\
&=&
\sqrt{-1}\cdot \int_{\pi^{-1}(y_2)} \left( y_2 \cdot \partial_{y_2} r + r \right)\, e^{1}\wedge e^{2} \wedge e^{3} \\
&=&
\sqrt{-1}\cdot \left( y_2 \cdot \partial_{y_2} R + R \right) ,
\end{eqnarray*}
}
where
$$ R(y_2) := \int_{\pi^{-1}(y_2)} r \, dx_1\wedge dy_1 \wedge dx_2 = \int_{\pi^{-1}(y_2)} r \, \mathrm{vol}_{\omega_T\lfloor \pi^{-1}(y_2)} . $$
By acting with $y_2\frac{\partial}{\partial y_2}$ on the constant quantity $\tau(y_2)$, we get the ordinary differential equation
$$ \Delta_{y_2^{-2}(dy_2)^2} R := y_2^2 \cdot \ddot R + y_2\cdot \dot R = 0 , $$
where $\dot R=\frac{dR}{dy_2}$, and $\Delta_{y_2^{-2}(dy_2)^2}$ is the Laplace-Beltrami operator on $\mathbb{S}^1$ with respect to the metric $y_2^{-2} (dy_2)^2$.
Therefore, it has no solution on $\mathbb{S}^1$ except than constants.

Therefore we get that

\begin{lem}\label{lem:gaud-sm}
On an Inoue-Bombieri surface of type $S_M$, if the Hermitian metric $\omega$ as in \eqref{eq:generic-metric} is Gauduchon, then $R:=\int_{\pi^{-1}(y_2)} r \, \mathrm{vol}_{\omega_{T}\lfloor \pi^{-1}(y_2)}$ does not depend on $y_2$.
\end{lem}

\subsection{Gauduchon metrics on Inoue-Bombieri surfaces of type \texorpdfstring{$S^+$}{S+-}}
If $\omega$ is Gauduchon, then
$$ d\int_{\pi^{-1}(y_2)} \overline\partial \omega = 0 , $$
therefore the function $\tau(y_2):=\int_{\pi^{-1}(y_2)} \overline\partial \omega$ is constant on $\mathbb S^1$.

We compute
\begin{eqnarray*}
   \overline\partial\omega &=&
   \left( \sqrt{-1} \cdot \partial_{\bar z} r \cdot (y_1-m\log y_2) + \sqrt{-1} \cdot \partial_{\bar w}r \cdot y_2 - \partial_{\bar z} u - \frac{1}{2} \cdot r \right) \, \varphi^1\wedge\bar\varphi^1\wedge\bar\varphi^2 \\
   && + \left( -\sqrt{-1} \cdot \partial_{\bar z}s - \partial_{\bar z}\bar u \cdot (y_1-m\log y_2) - \partial_{\bar w}\bar u \cdot  y_2 - \frac{m}{2} \cdot r+\frac{\sqrt{-1}}{2} \cdot u \right) \, \varphi^2\wedge\bar\varphi^1\wedge\bar\varphi^2 .
\end{eqnarray*}
\dan{
Therefore, by straightforward computations,
\begin{eqnarray*}
\tau(y_2)
&=& \int_{\pi^{-1}(y_2)} \left( 2 \cdot \partial_{\bar z} r \cdot (y_1-m\log y_2) + 2 \cdot \partial_{\bar w}r \cdot y_2 + 2\sqrt{-1}  \partial_{\bar z} u + \sqrt{-1}  \cdot r \right) \, e^1\wedge e^2\wedge e^3 \\
&=&
\int_{\pi^{-1}(y_2)} \left( \partial_{x_1} r \cdot (y_1-m\log y_2) + \partial_{x_2} r \cdot y_2 - \partial_{y_1}u \right) \cdot e^{1}\wedge e^{2}\wedge e^3 \\
&& + \sqrt{-1} \cdot \int_{\pi^{-1}(y_2)} \left( \partial_{y_1} r \cdot (y_1-m\log y_2) + \partial_{y_2}r \cdot y_2 + \partial_{x_1}u + r \right) \, e^1\wedge e^2\wedge e^3 \\
&=&
\int_{\pi^{-1}(y_2)} dr \wedge e^{1} \wedge e^{2} + \int_{\pi^{-1}(y_2)} du \wedge e^1 \wedge e^3 + \sqrt{-1} \cdot \int_{\pi^{-1}(y_2)} \left( \partial_{y_2}r \cdot y_2 + r \right) \, e^1\wedge e^2\wedge e^3 \\
&& - \sqrt{-1} \cdot \int_{\pi^{-1}(y_2)} (y_1-m\log y_2)\cdot dr \wedge e^1 \wedge e^3 + \sqrt{-1}\cdot \int_{\pi^{-1}(y_2)} du \wedge e^2 \wedge e^3 \\
&=&
\sqrt{-1} \cdot \int_{\pi^{-1}(y_2)} \partial_{y_2}r \cdot y_2 \cdot e^1\wedge e^2\wedge e^3 \\
&=& \sqrt{-1}\cdot y_2\cdot \partial_{y_2} R ,
\end{eqnarray*}
}
where
$$ R(y_2) := \int_{\pi^{-1}(y_2)} r \cdot e^1\wedge e^2\wedge e^3 = \int_{\pi^{-1}(y_2)} r \, \frac{1}{y_2} dx_1\wedge dy_1 \wedge dx_2 . $$
Hence we get the ordinary differential equation
$$ y_2 \cdot \ddot R + \dot R = 0 , $$
where $\dot R=\frac{dR}{dy_2}$.
It has no solution on $\mathbb{S}^1$ except than constants.

Therefore we get that

\begin{lem}\label{lem:gauduchon-s+-}
 On an Inoue-Bombieri surface of type $S^{+}$, if the Hermitian metric $\omega$ as in \eqref{eq:generic-metric} is Gauduchon, then $R:=\int_{\pi^{-1}(y_2)} r \, \mathrm{vol}_{\omega_{V}\lfloor \pi^{-1}(y_2)}$ does not depend on $y_2$.
\end{lem}

\section{Strongly leafwise flat forms on Inoue-Bombieri surfaces}\label{sec:slf-metrics}

\subsection{Strongly leafwise flat forms}
Recall that a real $(1,1)$-form $\eta$ on an Inoue-Bombieri surface $S$ is called {\em strongly leafwise flat} if
$$ \frac{\eta\wedge\alpha}{\omega_{TV}^2} \in\R^{>0} , $$
where the constant equals $\frac{\int_{S}\eta\wedge\alpha}{\int_{S}\omega_{TV}^2}$, and where $\omega_{TV}$ denotes either the Tricerri or the Vaisman-Tricerri metric according to the type of $S$.

The {\em $\partial\overline\partial$-class} of $\omega$ is given by
$$ \omega_u := \omega+\sqrt{-1}\partial\overline\partial u > 0 $$
varying $u\in\left.\mathcal{C}^\infty(S;\R)\middle\slash\R\right.$.

Looking for $u$ as above such that $\omega+\sqrt{-1}\partial\overline\partial u$, \dan{in the $\partial\overline\partial$-class of $\omega$}, is a strongly leafwise flat $(1,1)$-form is equivalent to solving the equation
\begin{equation}\label{eq:Deltau=G}
\Delta_{\mathcal D}u = G(\omega) ,
\end{equation}
where
$$ \Delta_{\mathcal{D}}u := \frac{\sqrt{-1}\,\partial\overline\partial u\wedge\alpha}{\omega_{TV}^2} $$
is a degenerate elliptic operator on $\mathcal{C}^\infty(S;\R)$, and where we set
$$ G(\omega) := -\frac{\omega\wedge\alpha}{\omega_{TV}^2}+\frac{\int_{S}\omega\wedge\alpha}{\int_{S}\omega_{TV}^2} . $$

More precisely, we have that
\begin{eqnarray*}
\Delta_{\mathcal{D}}u = \frac{1}{8y_2} \, \partial_{z\bar z}u &\qquad& \text{ in case }S_M, \\
\Delta_{\mathcal{D}}u = \frac{1}{8} \, \partial_{z\bar z}u &\qquad& \text{ in case }S^\pm,
\end{eqnarray*}
is the Laplacian along the leaves with respect to the Tricerri, respectively Vaisman-Tricerri metric.

For any given Hermitian metric $\omega$ on $S$, which we can write in the form \eqref{eq:generic-metric}, we have
\begin{equation}\label{g}
 G(\omega) = -\frac{1}{8} r + \frac{1}{8} \, \dashint_{S} r \omega_{TV}^2 .
\end{equation}

\subsection{Gauduchon obstruction}
Consider
\begin{eqnarray*}
\ker\Delta_{\mathcal D} &=& \pi^*\mathcal{C}^\infty(\mathbb{S}^1;\R) \\
&=& \left\{ \psi=\psi(y_2) \in \mathcal{C}^\infty(\R^{>0};\R) : \psi(\Lambda \cdot y_2)=\psi(y_2) \text{ for any }y_2 \right\} ,
\end{eqnarray*}
where $\Lambda=\lambda$ in case $S_M$, and $\Lambda=\gamma$ in case $S^+$.
Indeed, if $u\in\mathcal{C}^\infty(S;\R)$ is in the kernel of $\Delta_{\mathcal D}$, then the restriction of $u$ to each leaf is a bounded harmonic function on $\C$ or $\C^*$, whence constant. Since each leaf is dense in a fibre of $\pi\colon S\to \mathbb{S}^1$, then $u$ is the pull-back of a smooth function over $\mathbb S^1$.
Conversely, each such function is in the kernel of $\Delta_{\mathcal D}$.

\begin{lem}\label{lem:obstr}
Given any Hermitian metric $\omega$ on $S$, if there is a smooth function $u$ such that $\omega+\sqrt{-1}\,\partial\overline\partial u$ is strongly leafwise flat, then we must have that
\begin{equation}\label{necess}
\int_{S} \psi G(\omega)\omega_{TV}^2=0,
\end{equation}
for all $\psi\in \ker\Delta_{\mathcal D}$.
\end{lem}
\begin{proof}
Indeed, for $\psi\in \ker\Delta_{\mathcal D}$, we have
\begin{eqnarray*}
\int_{S} \psi G(\omega)\omega_{TV}^2 &=& \int_{S} \psi \Delta_{\mathcal{D}}u\omega_{TV}^2 = \int_{S}\psi \sqrt{-1}\,\partial\overline\partial u\wedge\alpha \\
&=& \int_{S}u \sqrt{-1}\,\partial\overline\partial \psi\wedge\alpha = \int_{S} u\Delta_{\mathcal{D}}\psi\omega_{TV}^2 = 0 ,
\end{eqnarray*}
yielding the statement.
\end{proof}

It is clear that every invariant metric on $S$ satisfies \eqref{necess}, since by invariance we have $G(\omega)=0$ in this case.

\begin{lem}\label{lem:obstr2}
The obstruction \eqref{necess} is satisfied for all $\psi\in \ker\Delta_{\mathcal D}$ if and only if we have
\begin{equation}\label{necessa}
\int_{\pi^{-1}(y_2)} G(\omega) \, \mathrm{vol}_{\omega_{TV}\lfloor \pi^{-1}(y_2)}=0,
\end{equation}
for all $y_2\in \mathbb{S}^1$.
\end{lem}
\begin{proof}
In one direction, suppose \eqref{necessa} holds and let $\psi$ be any element of $\ker\Delta_{\mathcal D}$, so $\psi$ is the pullback to $S$ of a smooth function on $\mathbb{S}^1$. Then we have
$$ \int_{S}\psi G(\omega)\omega_{TV}^2=c\int_{\mathbb{S}^1}\psi \left( \int_{\pi^{-1}(y_2)} G(\omega) \mathrm{vol}_{\omega_{TV}\lfloor \pi^{-1}(y_2)} \right)\frac{dy_2}{y_2}=0,$$
where $c$ is a numerical positive constant.

For the converse, if we had
$\int_{\pi^{-1}(q_0)} G(\omega) \, \mathrm{vol}_{\omega_T\lfloor \pi^{-1}(q_0)}>0$ for some $q_0\in \mathbb{S}^1$, then by continuity this would be true for all $q\in U$ for some  open subset $U\subset \mathbb{S}^1$ containing $q_0$, and then if we choose a smooth function $\psi$ on $\mathbb{S}^1$ which is nonnegative, compactly supported in $U$ and positive at $q_0$, we obtain that $\int_{S} \psi G(\omega)\omega_{TV}^2>0$, contradicting \eqref{necess}.
\end{proof}

The following result will be crucial to us.
\begin{lem}\label{lem:ostruction-slf}
Every Gauduchon metric on $S$ satisfies \eqref{necess}, and there exist (non-Gauduchon) Hermitian metrics on $S$ which do not satisfy \eqref{necess}.
\end{lem}
\begin{proof}
Given any Gauduchon metric $\omega$ on $S$ as in \eqref{eq:generic-metric}, the associated function $r$ on $S$ satisfies that
$$ R(y_2) := \int_{\pi^{-1}(y_2)} r \, \mathrm{vol}_{\omega_{TV}\lfloor \pi^{-1}(y_2)} $$
is a constant function of $y_2\in \mathbb{S}^1$, thanks to Lemmas \ref{lem:gaud-sm} and \ref{lem:gauduchon-s+-}. Thanks to \eqref{g}, this implies that
\begin{eqnarray}\label{necess2}
\int_{\pi^{-1}(y_2)} G(\omega) \, \mathrm{vol}_{\omega_{TV}\lfloor \pi^{-1}(y_2)}
&=& -\frac{1}{8}\cdot R(y_2) + \frac{1}{8} \cdot \dashint_S R(y_2) \cdot \omega_{TV}^2 = 0 ,
\end{eqnarray}
for all $y_2\in \mathbb{S}^1$.
The first statement thus follows from Lemma \ref{lem:obstr2}.

\smallskip

For the second statement, choose any smooth positive nonconstant function $r\colon \mathbb{S}^1\to\mathbb{R}$, and define
$$\omega:=\sqrt{-1}r\, \varphi^1\wedge\bar\varphi^1 + \sqrt{-1}\, \varphi^2\wedge\bar\varphi^2.$$
Then if we choose $\psi=r\in \ker\Delta_{\mathcal{D}}$, we have
$$\int_{S} \psi G(\omega)\omega_{TV}^2 = -\frac{1}{8} \int_{S}r^2\omega_{TV}^2+\frac{1}{8}\frac{\left(\int_{S}r\omega_{TV}^2\right)^2}{\int_{S}\omega_{TV}^2}<0,$$
because $r$ is nonconstant.
\end{proof}

\subsection{Strongly leafwise flat forms on Inoue-Bombieri surfaces \texorpdfstring{$S_M$}{SM}}
In this section, we explicitly solve equation \eqref{eq:Deltau=G} on $S_M$, so proving Theorem \ref{thm:main-thm} for the case of Inoue-Bombieri surfaces of type $S_M$.

\begin{thm}\label{main2}
Let $\omega$ be a Hermitian metric on $S_M$ which satisfies \eqref{necess}. Then there is a smooth function $u$ such that $\omega+\sqrt{-1}\,\partial\overline\partial u$ is a strongly leafwise flat $(1,1)$-form.
\end{thm}
\begin{proof}
Recall that there is a quotient map $p:\mathbb{T}^3\times [1,\lambda]\to S_M$, which identifies $(q,1)$ with $(\Psi(q),\lambda)$ for a certain diffeomorphism $\Psi$ of $\mathbb{T}^3$ induced by $f_0$.
Here $\mathbb{T}^3=\mathbb{R}^3/\Lambda$, where we use the coordinates $x_1,y_1,x_2$ on $\mathbb{R}^3$, and the lattice $\Lambda$ is spanned by $\partial_1$, $\partial_2$, $\partial_3$ which are given by
 \begin{equation}\label{eq:delta}
    \partial_1 :=
    \left(\begin{array}{c}\Re m_1\\ \Im m_1\\ \ell_1 \end{array}\right) , \qquad
    \partial_2 :=
    \left(\begin{array}{c}\Re m_2\\ \Im m_2\\ \ell_2 \end{array}\right) , \qquad
    \partial_3 :=
    \left(\begin{array}{c}\Re m_3\\ \Im m_3\\ \ell_3 \end{array}\right) ,
 \end{equation}
Pulling back $G(\omega)$ via $p$ we obtain a smooth function $g$ on $\mathbb{T}^3\times [1,\lambda]$ which thanks to \eqref{necessa} (which is equivalent to \eqref{necess} by Lemma \ref{lem:obstr2}) satisfies
\begin{equation}\label{avg}
\int_{\mathbb{T}^3\times \{y_2\}} g dx_1\wedge dy_1\wedge dx_2=0,
\end{equation}
for all $y_2\in [1,\lambda]$. The equation we wish to solve pulls back to the PDE
\begin{equation}\label{eq:Deltau=G2}
\frac{1}{32y_2}\left(\frac{\partial^2u}{\partial x_1^2}+\frac{\partial^2u}{\partial y_1^2}\right)=g,
\end{equation}
on $\mathbb{T}^3\times [1,\lambda]$. Since the differential operators $\frac{\partial}{\partial x_1}$ and $\frac{\partial}{\partial y_1}$ are in the $\mathbb{T}^3$ direction, it is natural to try to solve this PDE by using the Fourier  series expansion on each $\mathbb{T}^3\times \{y_2\}$ separately.
We write $\frac{\partial}{\partial x_1}$ and $\frac{\partial}{\partial y_1}$  with respect to the basis $\{\partial_1, \partial_2, \partial_3\}$ as
$$ \frac{\partial}{\partial x_1} = A_1\,\partial_1+A_2\,\partial_2+A_3\,\partial_3
   \qquad \text{ and } \qquad
   \frac{\partial}{\partial y_1} = B_1\,\partial_1+B_2\,\partial_2+B_3\,\partial_3 , $$
where
\begin{eqnarray*}
\left(\begin{array}{cc}A_1&B_1\\A_2&B_2\\A_3&B_3\end{array}\right) &=&
\left(\begin{array}{ccc}\Re m_1&\Re m_2&\Re m_3\\\Im m_1&\Im m_2&\Im m_3\\\ell_1&\ell_2&\ell_3\end{array}\right)^{-1}
\cdot \left(\begin{array}{cc}1&0\\0&1\\0&0\end{array}\right) \\
&=& \frac{1}{\varepsilon} \cdot \left(\begin{array}{ccc}
-\Im m_3\ell_2+\Im m_2\ell_3 & -\ell_3 \Re m_2+\ell_2\Re m_3 \\
\Im m_3\ell_1-\Im m_1\ell_3 & \ell_3\Re m_1-\ell_1\Re m_3 \\
-\Im m_2\ell_1+\Im m_1\ell_2 & -\ell_2\Re m_1+\ell_1\Re m_2
\end{array}\right) ,
\end{eqnarray*}
where
\begin{eqnarray*}
\varepsilon&=& \det \left(\begin{array}{ccc}\Re m_1&\Re m_2&\Re m_3\\\Im m_1&\Im m_2&\Im m_3\\\ell_1&\ell_2&\ell_3\end{array}\right) \\
&=& (-\Im m_3\ell_2+\Im m_2\ell_3)\Re m_1-(-\Im m_3\ell_1+\Im m_1\ell_3)\Re m_2\\
&&+(-\Im m_2\ell_1+\Im m_1\ell_2)\Re m_3.
\end{eqnarray*}

We compute the Laplacian along the leaves in terms of the above basis:
\begin{eqnarray*}
\Delta_{\mathcal D} &=& \frac{1}{8y_2}\cdot\frac{\partial^2}{\partial z\partial\bar z} = \frac{1}{32y_2} \left( \frac{\partial^2}{\partial x_1\partial x_1}+\frac{\partial^2}{\partial y_1\partial y_1} \right) \\
&=& \frac{1}{32 y_2} \left( \left(\sum_j A_j\partial_j\right)^2 + \left(\sum_j B_j\partial_j\right)^2 \right) \\
&=& \frac{1}{32 y_2} \left( \sum_j \left(A_j^2+B_j^2\right) \partial_{jj} + 2\,\sum_{j<h} \left(A_jA_h+B_jB_h\right) \partial_{jh} \right) \\
&=& \frac{1}{32y_2} \cdot \left(\begin{array}{ccc}\partial_1&\partial_2&\partial_3\end{array}\right) \cdot Z \cdot \left(\begin{array}{c}\partial_1\\\partial_2\\\partial_3\end{array}\right) ,
\end{eqnarray*}
where
\begin{eqnarray*}
Z &=& \left(\begin{array}{ccc}
A_{1}^{2} + B_{1}^{2} & A_{1} A_{2} + B_{1} B_{2} & A_{1} A_{3} + B_{1} B_{3} \\
A_{1} A_{2} + B_{1} B_{2} & A_{2}^{2} + B_{2}^{2} & A_{2} A_{3} + B_{2} B_{3} \\
A_{1} A_{3} + B_{1} B_{3} & A_{2} A_{3} + B_{2} B_{3} & A_{3}^{2} + B_{3}^{2}
\end{array}\right) \\
&=& \left(\begin{array}{rrr}
A_{1} & A_{2} & A_{3} \\
B_{1} & B_{2} & B_{3}
\end{array}\right)^t \cdot \left(\begin{array}{rrr}
A_{1} & A_{2} & A_{3} \\
B_{1} & B_{2} & B_{3}
\end{array}\right) ,
\end{eqnarray*}
which is semipositive definite of rank $2$, whose kernel is generated by the vector
\begin{equation}\label{boh}
 \left( \begin{array}{c} A_3B_2 - A_2B_3 \\ - A_3B_1 + A_1B_3 \\ A_2B_1 - A_1B_2 \end{array}\right)=-\frac{1}{\varepsilon}\left(\begin{array}{c}\ell_{1}\\\ell_{2}\\\ell_{3}\end{array}\right),
 \end{equation}
where the second equality follows by direct inspection.

We solve equation \eqref{eq:Deltau=G2} by considering the Fourier series expansion.
More precisely, consider the Fourier series of the datum in terms of the above basis for the torus:
$$ g(t,y_2) = \sum_{k\in\Z^3\backslash\{0\}} d_k(y_2) \cdot \exp\left(2\pi\sqrt{-1}\left\langle k|t\right\rangle\right) , $$
where $k=(k_1,k_2,k_3)$, and $t=(t_1,t_2,t_3)$ are the coordinates in the basis $(\partial_1,\partial_2,\partial_3)$, and $\left\langle k|t\right\rangle=k_1t_1+k_2t_2+k_3t_3$. The zero Fourier mode does not appear because of \eqref{avg}, while the other Fourier coefficients $d_k(y_2)\in\mathbb{C}$ depend smoothly on $y_2$.
Consider the Fourier series of the expected solution:
\begin{equation}\label{sol}
 u(t,y_2) = \sum_{k\in\Z^3\backslash\{0\}} a_k(y_2) \cdot \exp\left(2\pi\sqrt{-1}\left\langle k|t\right\rangle\right) .
 \end{equation}
Equation \eqref{eq:Deltau=G2} is equivalent to, for any $k\in\Z^3\backslash\{0\}$,
$$ a_k(y_2) = \left(2\pi\sqrt{-1}\right)^{-2} \cdot \frac{32y_2\cdot d_k(y_2)}{z_k} $$
where
$$ z_k := \left(\begin{array}{ccc}k_1&k_2&k_3\end{array}\right) \cdot 	Z \cdot \left(\begin{array}{c}k_1\\k_2\\k_3\end{array}\right) .
$$
For this to make sense, we need to show that $z_k\neq 0$ for all $k\in \Z^3\backslash\{0\}$, or equivalently that the kernel of the $3\times 3$ semipositive definite matrix $Z$
meets $\mathbb{Z}^3$ in $\{(0,0,0)\}$ only. As we said above, the kernel of $Z$ is $1$-dimensional spanned by the vector
$$v=\left(\begin{array}{c}\ell_{1}\\\ell_{2}\\\ell_{3}\end{array}\right).  $$
If we have $\nu v\in\mathbb{Z}^3$ for some $\nu\in\mathbb{R}$, then $(\nu\ell_1,\nu\ell_2,\nu\ell_3)\in \mathbb{Z}^3$ and the relations
$$\lambda\nu\ell_j=\sum_{k=1}^3M_{jk}\nu\ell_k,$$
where $M_{jk}\in\mathbb{Z}$, imply that $\nu=0$. Therefore we have shown that $z_k\neq 0$ for all $k\in \Z^3\backslash\{0\}$, and so \eqref{sol} defines a distributional solution $u$ of \eqref{eq:Deltau=G2} on $\mathbb{T}^3\times\{y_2\}$ for all $y_2$. Recall the well-known fact that a Fourier series on a torus of the form \eqref{sol} defines a smooth function if and only if its Fourier coefficients $a_k$ decay faster than any power $|k|^{-N},N>0,$ as $|k|\to\infty$. In particular, since our datum $g$ is smooth, the coefficients $d_k(y_2)$ satisfy this decay property, \dan{for any fixed $y_2$}. It follows that, exactly as in \cite{GW}, to show smoothness of $u(\cdot,y_2)$ on each fiber $\mathbb{T}^3\times\{y_2\}$, it is sufficient to show that
\begin{equation}\label{need}
|z_k|\geq \frac{C}{(k_1^2+k_2^2+k_3^2)^M},
\end{equation}
for some $C,M>0$ and for all $|k_1|, |k_2|, |k_3|$ sufficiently large.

To see this, we use the Liouville theorem \dan{on Diophantine approximation}:
if $x\in\mathbb{R}$ is irrational and algebraic of degree $d>1$, then there is a constant $C>0$ such that for all integers \dan{$(p,q)$, with $q>1$}, we have
\begin{equation}\label{eq:liouville-thm}
\left|x-\frac{p}{q}\right|\geq \frac{C}{q^d}.
\end{equation}
\dan{The property \eqref{eq:liouville-thm} says that $x$ is a {\em non-Liouville number}, or {\em strongly-dispersive} in the terminology of \cite{otiman-toma}.}

To deduce \eqref{need} from Liouville's Theorem, first note that $|z_k|$ is comparable to the square of the distance from the point $(k_1,k_2,k_3)$ to the real line spanned by $v$ (since the other two eigenvalues of $Z$ are positive). Recall that at least one of the ratios $\{\ell_i/\ell_j\}_{i\neq j}$ is well-defined and is irrational, and we may assume that it is $\ell_2/\ell_1$. Then this distance from $(k_1,k_2,k_3)$ to the span of $v$ is larger than or equal to the distance from $(k_1,k_2)$ to the line spanned by $(\ell_1,\ell_2)$, which up to a fixed constant equals
$$\left|k_2-\frac{\ell_2}{\ell_1}k_1\right|=|k_1|\left|\frac{k_2}{k_1}-\frac{\ell_2}{\ell_1}\right|.$$
But the ratio $\ell_2/\ell_1$ is an irrational algebraic number (of degree say $d$), and so Liouville's Theorem shows that
$$\left|\frac{k_2}{k_1}-\frac{\ell_2}{\ell_1}\right|\geq \frac{C}{|k_1|^d},$$
for a fixed constant $C>0$ and so we have shown that
$$|z_k|\geq \frac{C'}{k_1^{2(d-1)}}\geq \frac{C'}{(k_1^2+k_2^2+k_3^2)^{d-1}},$$
as claimed.

We have thus proved smoothness of $u$ on each $\mathbb{T}^3$ fiber, and smoothness in the variable $y_2$ now follows from the smoothness of the Fourier coefficients $a_k(y_2)$. We have thus obtained a smooth solution $u$ of \eqref{eq:Deltau=G2} on $\mathbb{T}^3\times[1,\lambda]$. Next we show that the solution $u$ descends to a smooth function on $S_M$, or equivalently that
$u(p,1)=u(\Psi(p),\lambda)$ for all $p\in\mathbb{T}^3$, where recall that $\Psi$ is induced by $f_0(z,w)=(\mu z,\lambda w)$. Of course the function $g$ on the right hand side of \eqref{eq:Deltau=G2} has this property. Let
$\tilde{u}(p,1)=u(\Psi(p),\lambda)$, which is a smooth function on $\mathbb{T}^3\times\{1\}$ which satisfies (writing $\mu=\mu_1+\sqrt{-1}\mu_2$ and $p=(x_1,y_1,x_2)$)
\begin{eqnarray*}
\Delta_{\mathcal{D}}\tilde{u}(p,1)&=&\frac{1}{32}\,\left(\frac{\partial^2}{\partial x_1^2}+\frac{\partial^2}{\partial y_1^2}\right)\left(u(\mu_1x_1-\mu_2y_1,\mu_1 y_1+\mu_2x_1,\lambda x_2)\right)\\
&=&|\mu|^2\,\frac{1}{16}\,\left(\frac{\partial^2u}{\partial x_1^2}+\frac{\partial^2u}{\partial y_1^2}\right)(\mu_1x_1-\mu_2y_1,\mu_1 y_1+\mu_2x_1,\lambda x_2)\\
&=&\frac{1}{16\,\lambda}\left(\frac{\partial^2u}{\partial x_1^2}+\frac{\partial^2u}{\partial y_1^2}\right)(\mu_1x_1-\mu_2y_1,\mu_1 y_1+\mu_2x_1,\lambda x_2)\\
&=&g(\Psi(p),\lambda) = g(p,1) =
\Delta_{\mathcal{D}}u(p,1),
\end{eqnarray*}
and hence $u(\cdot,1)$ and $\tilde{u}(\cdot,1)$ differ by a constant, which is in fact zero since both functions have integral zero on the torus. Therefore $u$ descends to a smooth solution of \eqref{eq:Deltau=G} on $S_M$, as desired.
\end{proof}

\subsection{Strongly leafwise flat forms on Inoue-Bombieri surfaces \texorpdfstring{$S^\pm$}{S+-}}
In this section, we explicitly solve equation \eqref{eq:Deltau=G} on $S^+$, proving Theorem \ref{thm:main-thm} for the remaining case of Inoue-Bombieri surfaces of type $S^\pm$.

\begin{thm}\label{main3}
Let $\omega$ be a Hermitian metric on $S^\pm$ which satisfies \eqref{necess}. Then there is a smooth function $u$ such that $\omega+\sqrt{-1}\,\partial\overline\partial u$ is a strongly leafwise flat $(1,1)$-form.
\end{thm}
\begin{proof}
As we mentioned earlier, it suffices to treat the case of $S^+$, since the case of $S^-$ reduces to $S^+$ by using the double cover. Indeed, the covering involution $\iota$ of $S^+$ (such that $S^-=S^+/\iota$) is of the form $\iota(z,w)=(\sqrt{\gamma}z,-w)$, see {\itshape e.g.} \cite[page 2130]{tosatti-weinkove-comp}, and this preserves the foliation. Pulling back $\omega$ to $S^+$ gives an $\iota$-invariant Hermitian metric $\hat{\omega}$ on $S^+$ which still satisfies \eqref{necess}, so we will find $u$ such that $\hat{\omega}+\sqrt{-1}\,\partial\overline\partial u$ is strongly leafwise flat. Therefore, so is $\hat{\omega}+\sqrt{-1}\,\partial\overline\partial \left(\frac{u+\iota^*u}{2}\right),$ which now descends to the desired strongly leafwise flat form on $S^-$.

From now on, we therefore work on an Inoue-Bombieri surface of type $S^+$. Our arguments begin by following the same outline as in the case of $S_M$, by looking at $S^+$ as a nilmanifold-bundle over $\mathbb{S}^1$, and then performing \dan{partial} Fourier series expansion on the nilmanifold which is a $\mathbb{T}^2$-bundle over $\mathbb{S}^1$. The proof of solvability of the ODEs that we will obtain will however be different from the case of $S_M$.
\dan{
For Fourier series on Heisenberg-type nilmanifolds, see {\itshape e.g.} \cite{auslander-tolimieri, deninger-singhof, richardson} as well as the very recent \cite{holt-zhang, holt-zhang-2,RS} which also apply this to study geometric questions.
}

Again, since the operator $\Delta_{\mathcal{D}}$ involves only derivatives in the directions of the fibers of the $X$-bundle structure $\pi:S^+\to\mathbb{S}^1$, we will first solve the PDE separately on each $\pi$-fiber, and later show that these piece together to a global solution. Since $\pi$ is a smooth fiber bundle, all fibers are diffeomorphic to the group-quotient $X_1=H(3;\R)/\Gamma'',$ the fiber over $y_2=1$, but now the operator $\Delta_{\mathcal D}$ actually depends on $y_2$. The fiber $X_{y_2}$ is equal to the quotient $X_{y_2} = \mathbb R^3 / \Gamma''$  in \eqref{quot} with fixed value of $y_2$, where $\Gamma'' = \langle (a_1, b_1, c_1), (a_2, b_2, c_2), (0,0, c_3)\rangle$ acts as the ``$y_2$-rescaled product''
$$ (a,b,c) \cdot (x_2,y_1,x_1) := (x_2+a, y_1+y_2 b, x_1+bx_2+c) . $$
Recall that, by Remark \ref{irraz}, neither $(a_1, a_2)$ nor $(b_1, b_2)$ has rational algebraic slope.

An arbitrary function $f$ on $X_{y_2}$ is identified with a function $f(x_2,y_1,x_1)$ on $\mathbb{R}^3$ which satisfies the periodicity conditions
\begin{eqnarray*}
f(x_2,y_1,x_1) &=& f(x_2+a_1,y_1+y_2 b_1,x_1+b_1x_2+c_1) \\
&=& f(x_2+a_2,y_1+y_2b_2,x_1+b_2x_2+c_2)\\
&=&f(x_2,y_1,x_1+c_3).
\end{eqnarray*}

{\rosso{
By taking advantage of the periodicity in the variable $x_1$, we take the Fourier series expansion of $f$ in the last variable:
$$ f(x_2,y_1,x_1) = \sum_{k\in\mathbb Z} f_k(x_2,y_1) \exp \left( 2 \pi \sqrt{-1} \frac{x_1}{c_3}k \right) , $$
where the coefficients $f_k$ are complex-valued functions such that
\begin{equation}\label{eq:reality}
f_k = \overline{f_{-k}} ,
\end{equation}
and satisfying the further periodicity conditions
\begin{eqnarray}\label{eq:periodicity-fourier}
f_k(x_2,y_1) &=& f_k(x_2+a_1, y_1+y_2b_1) \exp \left( 2\pi\sqrt{-1} \left( \frac{b_1}{c_3}x_2+\frac{c_1}{c_3} \right) k \right) \\
&=& f_k(x_2+a_2, y_1+y_2b_2) \exp \left( 2\pi\sqrt{-1} \left( \frac{b_2}{c_3}x_2+\frac{c_2}{c_3} \right) k \right) . \nonumber
\end{eqnarray}
In particular, for $k=0$, these rewrite as
\begin{eqnarray*}
f_0(x_2,y_1) &=& f_0(x_2+a_1, y_1+y_2b_1) \\
&=& f_0(x_2+a_2, y_1+y_2b_2) ,
\end{eqnarray*}
that is, $f_0$ is periodic with respect to the lattice
$$ \mathbb Z \left(\begin{matrix}a_1\\y_2b_1\end{matrix}\right) \oplus \mathbb Z \left(\begin{matrix}a_2\\y_2b_2\end{matrix}\right) , $$
and we can further expand  
$$ f_0(x_2,y_1) = \sum_{m,n\in\mathbb Z} f_{0,m,n} \exp \left( \frac{2\pi\sqrt{-1}}{y_2(a_1b_2-a_2b_1)} \left( y_2b_2x_2-a_2y_1 \right)m+ \frac{2\pi\sqrt{-1}}{y_2(a_1b_2-a_2b_1)} \left( -y_2b_1x_2+a_1y_1 \right) n \right) , $$
where $f_{0,m,n}\in\mathbb C$ are such that $f_{0,m,n}=\overline{f_{0,-m, -n}}$.
}}

{\rosso{
For simplicity of notation, we again denote the pullback of our function $32 G(\omega)$ to $\mathbb R^3$ by $g$, which can then be expanded as above
$$ g(x_2,y_1,x_1) = \sum_{k\in\mathbb Z} g_k(x_2,y_1) \exp \left( 2 \pi \sqrt{-1} \frac{x_1}{c_3}k \right) , $$
where the coefficients $g_k$ satisfy \eqref{eq:reality} and the periodicity conditions \eqref{eq:periodicity-fourier}, in particular,
$$ g_0(x_2,y_1) = \sum_{m,n\in\mathbb Z} g_{0,m,n} \exp \left( \frac{2\pi\sqrt{-1}}{y_2(a_1b_2-a_2b_1)} \left( y_2b_2x_2-a_2y_1 \right)m+ \frac{2\pi\sqrt{-1}}{y_2(a_1b_2-a_2b_1)} \left( -y_2b_1x_2+a_1y_1 \right) n \right) , $$
where crucially the zero mode $g_{0,0,0}$ vanishes, thanks to \eqref{necessa} (which is equivalent to the obstruction \eqref{necess} by Lemma \ref{lem:obstr2}).
}}

{\rosso{
The equation \eqref{eq:Deltau=G} $32\,\Delta_{\mathcal{D}}u=g$ reduces to the system
\begin{equation}\label{eq:ode-k}
\left(\frac{\partial}{\partial y_1}\right)^2 u_k(x_2,y_1) - 4\pi^2\frac{k^2}{c_3^2} u_k(x_2,y_1) = g_k(x_2,y_1) , \qquad \text{ for } k \in \mathbb Z,
\end{equation}
which are linear ordinary differential equations of second order in $y_1$.
}}

{\rosso{
Consider first the case $k \in \mathbb Z \setminus \{0\}$. It is clear that $g_k$ is smooth and bounded on $\mathbb R^2$, indeed it satisfies the periodicity conditions \eqref{eq:periodicity-fourier}. Then, for any fixed $x_2 \in \mathbb R$, we get a unique solution $u_k(x_2,\_)$ for \eqref{eq:ode-k} which is smooth in $(x_2,y_1)\in \mathbb R^2$ and bounded in $y_1$, see \cite[Proposition 8.2]{coppel}.
To apply the latter cited Proposition, one requires that the homogeneous equation
$$
\left(\begin{matrix} u \\ v \end{matrix}\right)'
= \left(\begin{matrix} 0 & 1 \\ \left(2\pi\cdot\sfrac{k}{c_3}\right)^2 & 0 \end{matrix}\right)
\cdot \left(\begin{matrix} u \\ v \end{matrix}\right) $$
has {\em bounded growth} and {\em exponential dichotomy} on the real line, as defined in \cite[pages 8 and 10]{coppel}. The bounded growth condition is readily verified, and since the above equation is autonomous, exponential dichotomy amounts to checking that the matrix has one eigenvalue with strictly positive real part and the other with strictly negative real part, see \cite[pages 10 and 19]{coppel}. This is indeed true, the eigenvalues being $\pm2\pi\cdot\sfrac{k}{c_3}\in\mathbb R\setminus\{0\}$.

The coefficients of the equation being real, then $u_k$ clearly satisfies the reality conditions \eqref{eq:reality}.
Moreover, since
\begin{eqnarray*}
\lefteqn{
\left(\frac{\partial}{\partial y_1}\right)^2 \left( u_k(x_2+a_j,y_1+y_2b_j)\exp \left( 2\pi\sqrt{-1} \left( \frac{b_j}{c_3}x_2+\frac{c_j}{c_3} \right) k \right) \right) }\\
&& - 4\pi^2\frac{k^2}{c_3^2} \left( u_k(x_2+a_j,y_1+y_2b_j) \exp \left( 2\pi\sqrt{-1} \left( \frac{b_j}{c_3}x_2+\frac{c_j}{c_3} \right) k \right) \right) \\
&=&
g_k(x_2+a_j,y_1+y_2b_j) \exp \left( 2\pi\sqrt{-1} \left( \frac{b_j}{c_3}x_2+\frac{c_j}{c_3} \right) k \right) = g(x_2,y_1) ,
\end{eqnarray*}
for $j\in\{1,2\}$, and by uniqueness of the bounded solution of \eqref{eq:ode-k}, we get that the solution $u_k$ satisfies the periodicity conditions \eqref{eq:periodicity-fourier}.
}}

{\rosso{
Look at $\Re u_k$, and let $y_1^{\max}$ be a maximum point for $\Re u_k(x_2,\_)$, once fixed $x_2\in\mathbb R$. Then
\begin{eqnarray*}
\max_{y_1\in\mathbb R} \Re u_k(x_2,y_1) &=& \Re u_k(x_2, y_1^{\max}) \\
&=&
\frac{1}{4\pi^2} \frac{c_3^2}{k^2} \left( \left(\frac{\partial}{\partial y_1}\right)^2 \Re u_k(x_2, y_1^{\max}) - \Re g_k(x_2,y_1^{\max}) \right) \\
&\leq& - \frac{1}{4\pi^2} \frac{c_3^2}{k^2} \Re g_k(x_2,y_1^{\max}) \leq \frac{1}{4\pi^2} \frac{c_3^2}{k^2} \max |g_k| .
\end{eqnarray*}
The same argument applies for $\Im u_k$, as well as at the minima points. This yields
\begin{equation}\label{eq:bound-k>0}
\max |u_k| \leq \frac{\sqrt{2}}{4\pi^2} \frac{c_3^2}{k^2} \max |g_k| .
\end{equation}
}}

{\rosso{
Consider now the case $k=0$. The differential equation \eqref{eq:ode-k} reduces to a system of algebraic equations:
\begin{equation}\label{eq:ode-mn}
-4\pi^2 u_{0,m,n} \left( \frac{n a_1 - m a_2}{y_2(a_1b_2-a_2b_1)} \right)^2 = g_{0,m,n} , \qquad \text{ for } m,n \in \mathbb Z .
\end{equation}
Since $\sfrac{a_1}{a_2}$ is irrational by Remark \ref{irraz}, then $n a_1 - m a_2\neq0$ for any $(m,n)\in\mathbb Z^2\setminus\{(0,0)\}$. Then, for $(m,n)\neq(0,0)$, we get the formal solution
$$ u_{0,m,n} = -\frac{1}{4\pi^2 } \left( \frac{y_2(a_1b_2-a_2b_1)}{n a_1 - m a_2} \right)^2 g_{0,m,n} . $$
Since $\sfrac{a_1}{a_2}$ is algebraic irrational, say of degree $d\geq 2$, then by the Liouville theorem we have
$$ |a_1n-a_2m|=|a_2| \cdot \left| \frac{a_1}{a_2}-\frac{m}{n} \right| \cdot |n| \geq C\cdot |a_2|\cdot|n|^{2-2d} $$
where $C$ is a positive constant depending on $\sfrac{a_1}{a_2}$. We get the estimate
\begin{equation}\label{eq:bound-k=0}
|u_{0,m,n}| \leq \frac{1}{4\pi^2C^2} \frac{|y_2(a_1b_2-a_2b_1)|^2}{|a_2|^2} \cdot |n|^{4d-4} \cdot |g_{0,m,n}| .\end{equation}
}}

{\rosso{
Since the datum $g$ is smooth, then, for $|k|\to+\infty$ and for all $N>0$, the Fourier coefficients $g_{k}$ decay faster than $|k|^{-2N}$ at any point, which implies that $\max|g_k|$ decay faster than $|k|^{-2N}$, respectively, the Fourier coefficients $g_{0,m,n}$ decay faster than $(m^2+n^2)^{-N}$. By \eqref{eq:bound-k>0} and \eqref{eq:bound-k=0}, the same decay property holds for $u_k$ and $u_{0,m,n}$. By this, and by the smoothness of the Fourier coefficients, we therefore get that
\begin{eqnarray*}
\lefteqn{ u(x_2,y_1,x_1) := \sum_{k\neq0} u_k(x_2,y_1) \exp\left(2\pi\sqrt{-1} \frac{x_1}{c_3}k\right) } \\
&& +\sum_{(m,n)\neq(0,0)} u_{0,m,n} \exp \left( \frac{2\pi\sqrt{-1}}{y_2(a_1b_2-a_2b_1)} \left( y_2b_2x_2-a_2y_1 \right)m+ \frac{2\pi\sqrt{-1}}{y_2(a_1b_2-a_2b_1)} \left( -y_2b_1x_2+a_1y_1 \right) n \right)
\end{eqnarray*}
is a smooth function on $X_{y_2}$.
(Note that the role of Liouville theorem to get smoothness here is less strong than before: compare with the ``greater regularity'' discussed in \cite[page 309]{richardson} for solutions of first order equations on the Heisenberg group with respect to the torus.)
}}

Lastly, smoothness of $u$ in the variable $y_2$ follows directly from smoothness of the Fourier coefficients of $g$ with respect to $y_2$, which implies the same smoothness for $u_{k}$ and $u_{0,m,n}$, by smooth dependence of the solution of the ODE.

Taking $y_2$ in the interval $[1,\gamma]$, we have thus obtained a smooth solution $u$ of \eqref{eq:Deltau=G2} on the subset of $\tilde{X}$ defined by $\pi_1^{-1}([1,\gamma]),$ which is diffeomorphic to $X\times[1,\gamma]$. Next we show that the solution $u$ descends to a smooth function on $S^+$, or equivalently that
$u(p,1)=u(\Psi(p),\gamma)$ for all $p\in X$, where recall that $\Psi$ is induced by $f_0(z,w)=(z+\mathbf t,\gamma \cdot w)$. Of course the function $g$ has this property. Let
$\tilde{u}(p,1)=u(\Psi(p),\gamma)$, which is a smooth function on $X\times\{1\}$ which satisfies (writing $\mathbf t=t_1+\sqrt{-1}t_2$)
\begin{eqnarray*}
\Delta_{\mathcal{D}}\tilde{u}(p,1)&=&\frac{1}{32}\,\left(\frac{\partial^2}{\partial x_1^2}+\frac{\partial^2}{\partial y_1^2}\right)\left(u(x_1+t_1,y_1+t_2,\gamma x_2)\right)\\
&=&\frac{1}{32}\,\left(\frac{\partial^2u}{\partial x_1^2}+\frac{\partial^2u}{\partial y_1^2}\right)(x_1+t_1,y_1+t_2,\gamma x_2)\\
&=&g(\Psi(p),\gamma) = g(p,1) =
\Delta_{\mathcal{D}}u(p,1),
\end{eqnarray*}
and hence $u(\cdot,1)$ and $\tilde{u}(\cdot,1)$ differ by a constant, which is in fact zero since both functions have integral zero on this fiber. Therefore $u$ descends to a smooth solution of \eqref{eq:Deltau=G} on $S^+$, as desired.
\end{proof}

Combining Theorems \ref{main2} and \ref{main3}  concludes the proof of Theorem \ref{thm:main-thm}.

Lastly, we give the proof of Corollary \ref{kor}. Given Theorem \ref{thm:main-thm}, this corollary is a more or less direct application of  \cite[Theorem 1.1]{fang-tosatti-weinkove-zheng}. The only thing to remark is that this latter result assumes that $\omega+\sqrt{-1}\,\partial\overline\partial u$ is a {\em Hermitian metric} which is strongly flat along the leaves, but a close inspection of its proof shows that all that is needed is that $\omega+\sqrt{-1}\,\partial\overline\partial u$ is a $(1,1)$-form which is strongly flat along the leaves (so it is positive definite in the leaves directions, but it need not be positive definite in all directions). Indeed, if we let
$$\tilde{\omega}(t)=(1-e^{-t})\omega_\infty+e^{-t}(\omega+\sqrt{-1}\,\partial\overline\partial u),$$
(as in \cite[Equation (2.7)]{fang-tosatti-weinkove-zheng}), then it is easy to see that these are Hermitian metrics on $S$ for all $t$ sufficiently large (using the Cauchy-Schwarz inequality to bound the terms involving $dz\wedge d\overline{w}$ and its conjugate). While this is not exactly the same as in \cite{fang-tosatti-weinkove-zheng}, where $\tilde{\omega}(t)$ are Hermitian metrics for all $t\geq 0$, this is irrelevant since we are only interested in the behavior of $\omega(t)$ for $t$ large. Therefore, with this change, the rest of the arguments in \cite{fang-tosatti-weinkove-zheng} go through verbatim, and this proves Corollary \ref{kor}.

\section{Higher order regularity}\label{sec:higher}
In this section we give the proof of Theorem \ref{thm:main-thm2}. In this section, $S$ will denote either an Inoue-Bombieri surface or a non-K\"ahler minimal properly elliptic surface (we refer to \cite[\S 8]{tosatti-weinkove-comp} and \cite{tosatti-weinkove-mathann} for the background material on these). On each of these surfaces there is an explicit Gauduchon metric discovered by Tricerri \cite{tricerri} or Vaisman \cite{vaisman}, that we will denote by $\omega_{TV}$ (for Inoue surfaces, these metrics were discussed earlier).

As in the statement of Theorem \ref{thm:main-thm2} we assume that the initial metric $\omega$ of the normalized Chern-Ricci flow \eqref{eq:CRF} is of the form $\omega=\omega_{TV}+\sqrt{-1}\partial\overline\partial \psi$ for some smooth function $\psi$ (in particular it is Gauduchon), and our goal is to prove uniform {\em a priori} estimates for $\omega(t)$ in $C^k(\omega)$ for all $k\geq 0$, and also to show that the curvature of $\omega(t)$ remains uniformly bounded for all $t\geq 0$.  As explained in the Introduction, we will adapt an idea first introduced in \cite{GTZ}, following the discussion in \cite[\S 5.14]{To} for the K\"ahler-Ricci flow, which dealt with the case of torus fibrations. Here, the role of the torus curves fibers will be played by the leaves of the canonical foliation on Inoue-Bombieri surfaces.

Let us first discuss the case when $S$ is of type $S_M$. Let $p:\mathbb{C}\times\mathbb{H}\to S$ be the universal covering map, with standard coordinates $(z,w)$ on $\mathbb{C}\times\mathbb{H}$, and let
$$\alpha=\frac{\sqrt{-1}}{4y_2^2}dw\wedge d\overline{w},\quad \beta=\sqrt{-1} y_2dz\wedge d\overline{z}.$$
The Tricerri metric is then given by $\omega_{TV}=4\alpha+\beta.$

Thanks to \cite[Theorem 2.4]{fang-tosatti-weinkove-zheng} (which we extended in Theorem \ref{thm:main-thm} to arbitrary initial Gauduchon metrics) we have
\begin{equation}\label{equiv}
C^{-1}(\alpha+e^{-t}\beta)\leq\omega(t)\leq C(\alpha+e^{-t}\beta),
\end{equation}
on $S\times[0,\infty)$. We also have that
$$\omega(t)=\omega_{\rm ref}(t)+\sqrt{-1}\partial\overline\partial\varphi(t),\quad \omega_{\rm ref}(t)=e^{-t}\omega_{TV}+(1-e^{-t})\alpha=(1+3e^{-t})\alpha+e^{-t}\beta,$$
for some smooth functions $\varphi(t)$ on $S$.

For $t\geq 0$ let
$\lambda_t:\mathbb{C}\times\mathbb{H}\to \mathbb{C}\times\mathbb{H}$ be given by
$$\lambda_t(z,w)=(ze^{t/2},w),$$
which is a ``stretching in the leaf directions'', analogous to the stretching first used in \cite{GTZ} (with the stretching done along abelian varieties which were the fibers of a fibration).
Observe that on $\mathbb{C}\times\mathbb{H}$, for all $t\geq 0$ we have
\begin{equation}\label{stretch}
\lambda_t^*p^*\alpha=p^*\alpha,\quad \lambda_t^*p^*\beta=e^tp^*\beta.
\end{equation}
These relations will give us the crucial property that the stretched reference metrics along the flow are smoothly comparable to Euclidean, as was the case for the semi-flat metrics that had to be carefully constructed in \cite{GTZ,HT}.

Then, for each $t$, fixed the metrics
$$\omega_t(s):=\lambda_t^*p^*\omega(s+t),\quad -t\leq s\leq 0,$$
on $\mathbb{C}\times\mathbb{H}$ satisfy
$$C^{-1}(\lambda_t^*p^*\alpha+e^{-s-t}\lambda_t^*p^*\beta)\leq \omega_t(s)\leq C(\lambda_t^*p^*\alpha+e^{-s-t}\lambda_t^*p^*\beta),$$
thanks to \eqref{equiv}. If we then restrict to $-1\leq s\leq 0$ then this implies
\begin{equation}
C^{-1}(\lambda_t^*p^*\alpha+e^{-t}\lambda_t^*p^*\beta)\leq \omega_t(s)\leq C(\lambda_t^*p^*\alpha+e^{-t}\lambda_t^*p^*\beta),
\end{equation}
and using \eqref{stretch} this implies
\begin{equation}\label{equi}
C^{-1}p^*\omega_{TV}\leq \omega_t(s)\leq Cp^*\omega_{TV},
\end{equation}
for all $t\geq 0, -1\leq s\leq 0$, and observe that $p^*\omega_{TV}$ is locally uniformly equivalent to a fixed Euclidean metric $\omega_E$ on $\mathbb{C}\times \mathbb{H}$. Furthermore, the metric $\omega_t(s)$ satisfy
\begin{equation}\label{scaled}
\frac{\partial}{\partial s}\omega_t(s)=-\mathrm{Ric}^{Ch}(\omega_t(s))-\omega_t(s), \quad -1\leq s\leq 0,
\end{equation}
and they are of the form
$$\omega_t(s)=\lambda_t^*p^*\omega_{\rm ref}(s+t)+\sqrt{-1}\partial\overline\partial(\lambda_t^*p^*\varphi(s+t)),$$
where
$$\lambda_t^*p^*\omega_{\rm ref}(s+t)=(1+3e^{-s-t})p^*\alpha+e^{-s}p^*\beta,$$
are Hermitian metrics which are uniformly smoothly bounded with respect to $\omega_E$, independent of $t\geq 0$ and $-1\leq s\leq 0$. We can thus apply the local higher order estimates of \cite{SW} and obtain that for each given compact set $K\Subset\mathbb{C}\times\mathbb{H}$ and $k\geq 0$ there are constants $C_{K,k}$ such that
$$\|\omega_t(s)\|_{C^k(K,g_E)}\leq C_{K,k},$$
for all $t\geq 0, -\frac{1}{2}\leq s\leq 0$. Setting $s=0$ we obtain
\begin{equation}\label{cinfty2}
\|\lambda_t^*p^*\omega(t)\|_{C^k(K,g_E)}\leq C_{K,k},
\end{equation}
and we still have
\begin{equation}\label{cinfty3}
\lambda_t^*p^*\omega(t)\geq C^{-1}\omega_E,
\end{equation}
on $K\times[0,\infty)$, from \eqref{equi} with $s=0$. From \eqref{cinfty2} and \eqref{cinfty3} we immediately see that
\begin{equation}\label{curvv2}
\sup_K |\mathrm{Rm}(\lambda_t^*p^*\omega(t))|_{\lambda_t^*p^*\omega(t)}\leq C,
\end{equation}
for all $t\geq 0$. If now $U\subset\mathbb{C}\times\mathbb{H}$ is the interior of a fundamental domain for the $\Gamma$-action on $\mathbb{C}\times\mathbb{H}$ (where $S=(\mathbb{C}\times\mathbb{H})/\Gamma$), then we have that $p$ is a biholomorphism between $U$ and an open dense subset of $S$ and so
$$\sup_{S}|\mathrm{Rm}(\omega(t))|_{\omega(t)}=\sup_U |\mathrm{Rm}(p^*\omega(t))|_{p^*\omega(t)}=\sup_{\lambda_{1/t}(U)} |\mathrm{Rm}(\lambda_t^*p^*\omega(t))|_{\lambda_t^*p^*\omega(t)},$$
where $\lambda_{1/t}$ is the inverse map of $\lambda_t$. But the compact sets $\lambda_{1/t}(U)$ are all contained in a fixed compact set $K\Subset\mathbb{C}\times\mathbb{H}$, and so
from \eqref{curvv2} we obtain that the curvature of $\omega(t)$ remains uniformly bounded for all times.

Next, it is easy to check working in the coordinates $(z,w)$ that \eqref{cinfty2} implies that
$$\|p^*\omega(t)\|_{C^k(K,g_E)}\leq C_{K,k},$$
and in fact \eqref{cinfty2} is a much stronger bound. Therefore taking $K$ again to be the closure of a fundamental domain, this estimate implies
$$\|\omega(t)\|_{C^k(S,\omega)}\leq C_k,$$
as desired.

We now discuss the case of $S^+$. The universal cover map is again denoted by $p:\mathbb{C}\times\mathbb{H}\to S$ with standard coordinates $(z,w)$ on $\mathbb{C}\times\mathbb{H}$, and let
$$\alpha=\frac{\sqrt{-1}}{4y_2^2}dw\wedge d\overline{w},\quad \beta=\sqrt{-1} \left(dz-\frac{y_1-m\log y_2}{y_2}dw\right)\wedge \left(d\overline{z}-\frac{y_1-m\log y_2}{y_2}d\overline{w}\right).$$
The Vaisman-Tricerri metric is then given by $\omega_{TV}=4\alpha+\beta.$ Defining the stretch maps $\lambda_t$ as before, we now have
$$\lambda_t^*p^*\alpha=p^*\alpha,\quad \lambda_t^*p^*\beta=e^t\sqrt{-1} \left(dz-\frac{y_1-e^{-\frac{t}{2}}m\log y_2}{y_2}dw\right)\wedge \left(d\overline{z}-\frac{y_1-e^{-\frac{t}{2}}m\log y_2}{y_2}d\overline{w}\right),$$
and so if we define the reference metrics
$\omega_{\rm ref}(t)=e^{-t}\omega_{TV}+(1-e^{-t})\alpha=(1+3e^{-t})\alpha+e^{-t}\beta,$ then again
$$\lambda_t^*p^*\omega_{\rm ref}(s+t)=(1+3e^{-s-t})p^*\alpha+e^{-s}p^*\beta,$$
are Hermitian metrics which are uniformly smoothly bounded with respect to $\omega_E$, independent of $t\geq 0$ and $-1\leq s\leq 0$, and the argument proceeds exactly as earlier. The case when $S$ is of type $S^-$ is easily reduced to the case of $S^+$ by passing to a double cover (cf. \cite[\S 7]{tosatti-weinkove-comp} and \cite{fang-tosatti-weinkove-zheng}).

Lastly, let $S$ be a minimal non-K\"ahler properly elliptic surface. Then its universal cover of $S$ is again $\mathbb{C}\times\mathbb{H}$ but it is more convenient to work instead with $\mathbb{C}^*\times\mathbb{H}$ via the map $(z,w)\mapsto (e^{-z/2},w)$. Then we have a holomorphic covering $p:\mathbb{C}^*\times\mathbb{H}\to S$, and if we denote by $u=e^{-z/2}$ the coordinate on $\mathbb{C}^*$ then the stretching maps $\lambda_t$ as earlier become
$\lambda_t:\mathbb{C}^*\times\mathbb{H}\to \mathbb{C}^*\times\mathbb{H}, \lambda_t(u,w)=(u^{e^{t/2}},w)$. Let
$$\alpha=\frac{\sqrt{-1}}{4y_2^2}dw\wedge d\overline{w},\quad \beta=\sqrt{-1} \left(-\frac{2du}{u}+\frac{idw}{y_2}\right)\wedge \left(-\frac{2d\overline{u}}{\overline{u}}-\frac{id\overline{w}}{y_2}\right).$$
The Vaisman metric is then given by $\omega_{TV}=4\alpha+\beta.$ We now have
$$\lambda_t^*p^*\alpha=p^*\alpha,\quad \lambda_t^*p^*\beta=e^t\sqrt{-1}  \left(-\frac{2du}{u}+e^{-\frac{t}{2}}\frac{idw}{y_2}\right)\wedge \left(-\frac{2d\overline{u}}{\overline{u}}-e^{-\frac{t}{2}}\frac{id\overline{w}}{y_2}\right),$$
and so if we define the reference metrics
$\omega_{\rm ref}(t)=e^{-t}\omega_{TV}+(1-e^{-t})\alpha=(1+3e^{-t})\alpha+e^{-t}\beta,$ then again
$$\lambda_t^*p^*\omega_{\rm ref}(s+t)=(1+3e^{-s-t})p^*\alpha+e^{-s}\sqrt{-1}  \left(-\frac{2du}{u}+e^{-\frac{t}{2}}\frac{idw}{y_2}\right)\wedge \left(-\frac{2d\overline{u}}{\overline{u}}-e^{-\frac{t}{2}}\frac{id\overline{w}}{y_2}\right),$$
are Hermitian metrics which are uniformly smoothly bounded with respect to $\omega_E$, independent of $t\geq 0$ and $-1\leq s\leq 0$, and the argument proceeds exactly as earlier.

\begin{rmk}
Instead of the local higher order estimates of \cite{SW}, we could have also used those in \cite{Chu} (which are the parabolic version of the results of \cite{TWWY}). For this, one writes the Chern-Ricci flow equation \eqref{scaled} as a scalar parabolic complex Monge-Amp\`ere equation on $\mathbb{C}\times\mathbb{H}$ for the potential $\lambda_t^*p^*\varphi(s+t)$, which is uniformly bounded in $L^\infty$ by the results in \cite{fang-tosatti-weinkove-zheng}. Nevertheless, in either of the two approaches it is important that the metrics $\lambda_t^*p^*\omega_{\rm ref}(s+t)$ are uniformly locally smoothly bounded for $t\geq 0, -1\leq s\leq 0$, and this is not the case in general if the initial metric $\omega$ is not assumed to be in the $\partial\overline\partial$-class of $\omega_{TV}$, even if it is assumed to be strongly flat along the leaves.
\end{rmk}

\begin{rmk}
In the case of minimal non-K\"ahler properly elliptic surfaces discussed above, it is also shown in \cite{tosatti-weinkove-mathann} that for every smooth fiber $E_w$, we have that $e^t\omega(t)|_{E_w}$ converges in the $C^1$ topology to the unique flat metric on $E_w$ cohomologous to $\omega_0|_{E_w}$. This convergence can be improved to $C^\infty$ by adapting an argument from \cite{TZ} as follows. Up to passing to a finite cover, we may assume that $S$ is an elliptic bundle, and pick a local trivialization $E\times B$ where $B$ is a ball in $\mathbb{C}$ and $E$ is an elliptic curve. Define stretching maps $\mu_t:E\times B\to E\times B$ by
$$\mu_t(z,w)=(z,we^{-t/2}),$$
and on $E\times B$ define
$$\omega_t(s)=e^t\mu_t^*\omega(se^{-t}+t), \quad -1\leq s\leq 0,$$
which are Hermitian metrics that solve
\begin{equation}
\frac{\partial}{\partial s}\omega_t(s)=-\mathrm{Ric}^{Ch}(\omega_t(s))-e^{-t}\omega_t(s), \quad -1\leq s\leq 0,
\end{equation}
for all $t\geq 0$, and they are of the form
$$\omega_t(s)=e^t\mu_t^*\omega_{\rm ref}(se^{-t}+t)+\sqrt{-1}\partial\overline\partial(e^t\mu_t^*\varphi(se^{-t}+t)),$$
where $\omega_{\rm ref}(t)=e^{-t}\omega_0+(1-e^{-t})\omega_\infty$, and $\omega_\infty$ is K\"ahler-Einstein, and by \cite[Thm. 5.1]{tosatti-weinkove-mathann} we have
\begin{equation}
C^{-1}e^t\mu_t^*\omega_{\rm ref}(se^{-t}+t)\leq \omega_t(s)\leq Ce^t\mu_t^*\omega_{\rm ref}(se^{-t}+t),\quad -1\leq s\leq 0, t\geq 0.
\end{equation}
One then checks as in \cite[Pf. of Thm. 1.1]{TZ} that $e^t\mu_t^*\omega_{\rm ref}(se^{-t}+t)$ are uniformly comparable to a fixed metric and smoothly uniformly bounded, for $-1\leq s\leq 0$ and $t\geq 0$, and hence the local higher order estimates of \cite{SW} give us uniform higher order estimates for $\omega_t(s), -\frac{1}{2}\leq s\leq 0$, on $E\times B$ (up to shrinking $B$ slightly). Setting $s=0$ this gives higher order estimates for $e^t\mu_t^*\omega(t)$, and restricting to any fiber $E\times \{w\}$ the maps $\mu_t$ are the identity, and we are done.
\end{rmk}

\end{document}